\documentclass[11pt]{article}

\usepackage[utf8]{inputenc}
\usepackage[T1]{fontenc}
\usepackage{amsmath,amssymb,amsthm,mathtools}
\usepackage{enumitem}
\usepackage[margin=1in]{geometry}
\usepackage[colorlinks=true,linkcolor=blue,citecolor=blue,urlcolor=blue]{hyperref}

\newtheorem{theorem}{Theorem}[section]
\newtheorem{proposition}[theorem]{Proposition}
\newtheorem{lemma}[theorem]{Lemma}
\newtheorem{corollary}[theorem]{Corollary}
\newtheorem{conjecture}[theorem]{Conjecture}
\theoremstyle{definition}
\newtheorem{definition}[theorem]{Definition}
\newtheorem{problem}[theorem]{Problem}
\theoremstyle{remark}
\newtheorem{remark}[theorem]{Remark}

\newcommand{\Fin}{\mathrm{Fin}}
\newcommand{\ED}{\mathcal{ED}}
\newcommand{\HH}{\mathcal H}
\newcommand{\GG}{\mathcal G}
\newcommand{\PC}{\mathcal{PC}}
\newcommand{\Sc}{\mathcal C}
\newcommand{\Pow}{\mathcal P}
\newcommand{\concat}{\mathbin{{}^\frown}}
\newcommand{\restr}{\mathbin{\upharpoonright}}
\newcommand{\rk}{\operatorname{rk}}
\newcommand{\suc}{\operatorname{Succ}}
\newcommand{\Fat}{\operatorname{Fat}}
\newcommand{\HomC}{\operatorname{Hom}}
\newcommand{\sh}{\operatorname{sh}}
\newcommand{\cfrak}{\mathfrak c}
\newcommand{\Fw}{\Fin_\omega}
\newcommand{\Fwp}{\Fin'_\omega}
\newcommand{\Hw}{\HH_{<\omega}}

\title{Tree-derived ideals: Fubini iterations, limit amalgamations,
and Kat\v etov obstructions}
\author{Jos\'e de Jes\'us Pelayo G\'omez\thanks{Independent
researcher. Email: \texttt{pelayuss@gmail.com}.}}
\date{July 16, 2026}

\begin{document}
\maketitle

\begin{abstract}
We develop a machinery for deriving ideals on a countable set from a
partition of $\omega$ indexed by the tree $\omega^{<\omega}$: a
derivative operator on trace trees, parametrized by an auxiliary
ideal $\mathcal J$, whose iterates produce a strict transfinite
hierarchy $\HH^{\mathcal J}_\alpha$ of proper ideals on $\omega$,
tall from level one onward. The machinery is canonical --- all
derived objects are independent of the chosen partition --- and its
finite levels recover exactly the Fubini powers:
$\HH_n\cong\Fin^{\otimes(n+1)}$, anchored on a single copy of
$\omega$. Three blocks of results demonstrate its scope.
\emph{Structure}: presentation independence and local homogeneity,
the Fubini recursion, descriptive complexity
($\mathbf\Pi^1_1$-completeness of the full hierarchy), and the nested
limit $\Hw=\bigcup_n\HH_n$, a new amalgamation of all finite Fubini
powers. \emph{Obstruction}: writing $\Fw$ for Kwela's canonical
inductive limit and $\Fwp$ for his independent-partitions limit, we
prove
\[
 \Fw\ \not\leq_K\ \Hw ,
\]
although $\Fwp\sqsubseteq\Hw$ and every finite coherent fragment of
such a reduction is realizable over $\Hw$ --- a genuinely infinitary
failure of compactness. The proof introduces the \emph{essential
depth} of a function with respect to the tree, an invariant monotone
along Kat\v etov reductions of $\Fin\otimes\Fin$, and yields the
sharp non-extension bound $N(m)=m+2$; no computation of the Borel
rank of the limit $\Hw$ is involved. It follows that $\Hw$ contains
no isomorphic copy of $\Fw$, and, via $\Fwp\sqsubseteq\Hw$, we obtain
a new proof of the Kat\v etov strengthening $\Fw\not\leq_K\Fwp$ of
Kwela's Theorem~4.6. \emph{Applications}: the chromatic ideals
$\GG_k$, which lie inside the second derived level but not inside the
first, ordered by inclusion exactly as divisibility while the
Kat\v etov order recodes the arithmetic ($k\geq\ell$ implies
$\GG_k\leq_{KB}\GG_\ell$, exactly so among radial reductions); the
orthogonality of Cantor--Bendixson rank and derivative rank; and the
quotient $\Pow(\omega)/\HH$, a $\sigma$-closed reduced power
satisfying $\mathbb B\cong\mathbb B^\omega/\Fin$, containing
$\Pow(\omega)/\Fin$ regularly and, under CH, forcing-equivalent to
$(\Pow(\omega)/\Fin)^+$.
\end{abstract}



\section{Introduction}\label{sec:intro}

Borel ideals on countable sets are a standard testing ground for the
combinatorics of the reals, and the Kat\v etov order $\leq_K$ is the
comparison tool that organizes them: it calibrates Ramsey-type
properties, destructibility by forcing, and critical ideals for
classes of spaces \cite{HrusakEtAl,CDU,FKK}. Within this landscape
the finite Fubini powers $\Fin^{\otimes n}$ play the role of
canonical benchmarks: they are the prototypical ideals of separation
rank $n$ \cite{DebsSaintRaymond}, they have the strong rigidity
property that any Kat\v etov reduction to an ideal can be upgraded to
an isomorphic copy inside it \cite{Barbarski}, and the question of
how they can be amalgamated into a single limit object of rank
$\omega$ is the subject of Kwela's theory of inductive limits of
ideals \cite{Kwela}.

This paper develops a machinery for producing and analyzing such
objects from a single combinatorial device: a partition
$\langle A_s:s\in\omega^{<\omega}\rangle$ of $\omega$ indexed by the
tree $\omega^{<\omega}$, introduced in \cite{Pelayo}. A tree-shaped
presentation of $\omega$ is not by itself a contribution --- any
countable structure transports along a bijection. What we claim as
new is the machinery built on it: (i) a \emph{derivative operator}
$D_{\mathcal J}$ on trace trees, parametrized by an arbitrary
auxiliary ideal $\mathcal J\supseteq\Fin$; (ii) its
\emph{presentation invariance} (Proposition~\ref{prop:independence}),
which makes every derived object canonical rather than an artifact of
the chosen partition; (iii) the \emph{exact recovery of Fubini
iterations} at the finite levels (Theorem~\ref{thm:fubini}), which
anchors all powers $\Fin^{\otimes(n+1)}$ coherently on a single copy
of $\omega$; (iv) a \emph{new limit object}, the nested amalgamation
$\Hw=\bigcup_n\HH_n$; and (v) the invariants the machinery supports,
above all the \emph{essential depth} of a function with respect to
the tree (Definition~\ref{def:depth}).

The obstruction block is motivated by Kwela's inductive limits
\cite{Kwela}. Kwela constructs the canonical inductive limit $\Fw$ of
the Fubini powers along suffix projections, and a second
amalgamation $\Fwp$ built from an independent family of partitions;
his Theorem~4.6 shows that $\Fwp$ contains no isomorphic copy of
$\Fw$, and his Conjecture~5.3 asks for analogous universal objects at
every countable limit ordinal --- with the case $\alpha=\omega$
already settled in \cite{Kwela} by $\Fwp$ itself (see
Remark~\ref{rem:kwelascope} for the precise scope of the present
results relative to that conjecture). The nested limit $\Hw$ is a
third natural amalgamation of the same finite powers, and comparing
the three is the obstruction problem this paper solves on one side:
$\Fw\not\leq_K\Hw$ (Theorem~\ref{thm:main}).

Two features of the proof deserve emphasis. First, the failure is
genuinely infinitary: every finite coherent fragment of a would-be
reduction is realizable over $\Hw$
(Proposition~\ref{prop:fragments}), so no argument through a fixed
finite number of levels can succeed; what fails is the simultaneous
coherence of all levels, and Theorem~\ref{thm:nonext} quantifies the
failure exactly, with the sharp bound $N(m)=m+2$. Second, the method
is new in this context: Kwela's proof of Theorem~4.6 proceeds by a
Borel-rank surgery that exploits the independence of the partitions
generating $\Fwp$, an independence that the nested partition
$\langle A_s\rangle$ conspicuously lacks. Our proof replaces the rank
computation by the combinatorial invariant of essential depth,
monotone along reductions of $\Fin\otimes\Fin$
(Theorem~\ref{thm:monotone}); no computation of the Borel rank of the
limit $\Hw$ is involved. Composing the main theorem with
$\Fwp\sqsubseteq\Hw$ yields an independent proof of the Kat\v etov
strengthening $\Fw\not\leq_K\Fwp$ of Kwela's theorem
(Corollary~\ref{cor:kwela46}).

The applications block demonstrates the scope of the machinery on
three fronts: a family of $F_\sigma$ ideals canonically attached to
the tree (the chromatic ideals $\GG_k$, which cannot be defined
without the partition and which the derived hierarchy measures
exactly), the interaction of the derivative rank with the classical
Cantor--Bendixson rank (they are orthogonal), and the quotient
algebra $\Pow(\omega)/\HH$ (a $\sigma$-closed reduced power with a
regular copy of $\Pow(\omega)/\Fin$). The results are organized in
three blocks.

\medskip\noindent\textbf{I. Construction and structure.}

\begin{enumerate}[label=(\Alph*)]
 \item \textbf{Rigidity and homogeneity} (Section~\ref{sec:prelim}):
 any two tree partitions are isomorphic via a level-preserving
 bijection of $\omega$; all objects of the paper are
 presentation-independent, localized quotients are isomorphic to the
 global one, and $\mathbb B\cong\mathbb B^\omega/\Fin$
 (Corollary~\ref{cor:selfref}).
 \item \textbf{The derived hierarchy} (Section~\ref{sec:DJ}): for
 every auxiliary ideal $\mathcal J\supseteq\Fin$, the derivative
 $D_{\mathcal J}$ produces a strict hierarchy of $\omega_1$ ideals
 whose successor steps are exact Fubini sums, whose positivity is
 witnessed by $\mathcal J$-positive Laver trees, and whose union is
 $\mathbf\Pi^1_1$-complete for Borel $\mathcal J$. The case
 $\mathcal J=\Fin$ recovers the powers $\Fin^{\otimes n}$ anchored on
 a single copy of $\omega$.
\end{enumerate}

\medskip\noindent\textbf{II. The obstruction theorem.}

\begin{enumerate}[label=(\Alph*),resume]
 \item \textbf{Main theorem} (Section~\ref{sec:main}):
 $\Fw\not\leq_K\Hw$, although $\Fwp\sqsubseteq\Hw$ and every finite
 coherent fragment is realizable (Proposition~\ref{prop:fragments}).
 The bound of Theorem~\ref{thm:nonext} is sharp. The main theorem also
 implies that $\Hw$ contains no isomorphic copy of $\Fw$; by
 \ref{K:katFw}, the two non-embedding statements are equivalent. We
 also prove that $\Hw$ has property $\mathrm{Kat}$
 (Theorem~\ref{thm:katH}), so the remaining comparisons out of $\Hw$
 may be phrased in terms of copies. Combining $\Fwp\sqsubseteq\Hw$
 with the main theorem gives an alternative route to the known
 Kat\v etov strengthening $\Fw\not\leq_K\Fwp$ of Kwela's
 Theorem~4.6.
\end{enumerate}

\medskip\noindent\textbf{III. Structural applications.}

\begin{enumerate}[label=(\Alph*),resume]
 \item \textbf{Chromatic ideals} (Sections~\ref{sec:Gk}
 and~\ref{sec:katGk}): the ideals $\GG_k$ generated by the homogeneous
 sets of $c_k(\{x,y\})=\Delta(b_x,b_y)\bmod k$ --- a family the
 derived hierarchy locates inside its second level but not inside its
 first (Theorems~\ref{thm:GH2} and~\ref{thm:H1PC}) --- satisfy
 $\GG_k\subseteq\GG_\ell\Leftrightarrow\ell\mid k$, while
 $k\geq\ell$ implies $\GG_k\leq_{KB}\GG_\ell$; within radial
 reductions the criterion $k\geq\ell$ is exact. The first open case of
 the conjectured classification is $\GG_2\leq_K\GG_3$.
 \item \textbf{Orthogonality of ranks} (Section~\ref{sec:CB}): the
 Cantor--Bendixson rank of the branch body and the rank of the
 derivative $D_{\mathcal J}$ measure orthogonal phenomena.
 \item \textbf{The quotient} (Section~\ref{sec:quotient}):
 $\Pow(\omega)/\HH$ decomposes exactly as the reduced product of its
 localizations, is $\sigma$-closed, contains $\Pow(\omega)/\Fin$
 regularly, has continuum-sized antichains of constant trace and,
 under CH, is forcing-equivalent to $(\Pow(\omega)/\Fin)^+$.
\end{enumerate}

\section{Preliminaries}\label{sec:prelim}

\subsection{The partition and the branch dictionary}

\begin{definition}\label{def:partition}
A \emph{tree partition} of $\omega$ is a family
$\langle A_s:s\in\omega^{<\omega}\rangle$ such that
$A_\varnothing=\omega$, each $A_s$ is infinite,
$\{A_{s\concat n}:n\in\omega\}$ is a partition of $A_s$, and any two
distinct points are separated at some level.
\end{definition}

For $x\in\omega$ there is a unique branch $b_x\in\omega^\omega$ with
$x\in A_{b_x\restr n}$ for all $n$; the map $x\mapsto b_x$ is
injective and its image $B=\{b_x:x\in\omega\}$ meets every cylinder
$[s]$ in an infinite set. For $p\neq q$ we write
$\Delta(p,q)=\min\{n:p(n)\neq q(n)\}$ and
$\Delta(x,y)=\Delta(b_x,b_y)$. For $X\subseteq\omega$,
\[
 T_X=\{s:X\cap A_s\neq\varnothing\},\qquad
 T_X^\infty=\{s:|X\cap A_s|=\omega\}.
\]

\begin{lemma}[Realization Lemma]\label{lem:realization}
Choose points $p_{s,n}\in A_{s\concat n}$, globally distinct, and set
$P_s=\{p_{s,n}:n\in\omega\}$, $X_S=\bigcup_{s\in S}P_s$. If
$S\subseteq\omega^{<\omega}$ is a tree, then $T^\infty_{X_S}=S$.
Moreover $S\mapsto X_S$ is continuous.
\end{lemma}

\begin{proof}
The points exist: choose them by recursion along a fixed enumeration
of the pairs $(s,n)$, at each step picking a point of
$A_{s\concat n}$ different from the finitely many points chosen so
far (possible, as every cell is infinite).

Let $u\in\omega^{<\omega}$. If $u\in S$, then
$P_u\subseteq A_u$ is infinite and $P_u\subseteq X_S$, so
$u\in T^\infty_{X_S}$. If $u\notin S$, consider a point
$p_{s,n}\in X_S\cap A_u$ with $s\in S$; then $s\concat n$ and $u$ are
comparable. If $u\subseteq s\concat n$, then either $u\subseteq s$,
which is impossible ($S$ is downward closed and $u\notin S$), or
$u=s\concat n$, which happens for at most one pair $(s,n)$, namely
$s=u\restr(|u|-1)$, $n=u(|u|-1)$. If $s\concat n\subsetneq u$, then
$A_u\subseteq A_{s\concat n}$ receives at most the single point
$p_{s,n}$, and there are at most $|u|$ such pairs. Altogether
$X_S\cap A_u$ is finite and $u\notin T^\infty_{X_S}$.

Continuity: a fixed point of $\omega$ belongs to $X_S$ if and only if
it is some $p_{s,n}$ and $s\in S$, a condition on a single
coordinate of (the characteristic function of) $S$.
\end{proof}

The next lemma is the rigidity tool of the paper. Its second clause
--- finite support of the local successor permutations --- is what
makes the derived objects of Section~\ref{sec:DJ} invariant for an
\emph{arbitrary} auxiliary ideal $\mathcal J\supseteq\Fin$: a
permutation of $\omega$ with finite support changes every subset of
$\omega$ by a finite set only, and hence preserves every such
$\mathcal J$.

\begin{lemma}[Back-and-forth Lemma]\label{lem:backforth}
Let $B,B'\subseteq\omega^\omega$ be countable and such that every
cylinder meets each of them in an infinite set. Then there is a
bijection $h:B\to B'$ preserving $\Delta$:
\[
 \Delta(h(p),h(q))=\Delta(p,q)\qquad(p\neq q).
\]
Moreover, $h$ may be chosen so that the induced level-preserving
bijection $\tau$ of $\omega^{<\omega}$, defined by
$\tau(p\restr n)=h(p)\restr n$, has the following property: for every
node $t$, the \emph{successor permutation}
$\pi_t:\omega\to\omega$, determined by
$\tau(t\concat n)=\tau(t)\concat\pi_t(n)$, has finite support (indeed
support of size at most $2$).
\end{lemma}

\begin{proof}
First note that $\tau$ is well defined for any $\Delta$-preserving
bijection $h$: if $p\restr n=q\restr n$ then $\Delta(p,q)\geq n$,
hence $\Delta(h(p),h(q))\geq n$ and $h(p)\restr n=h(q)\restr n$.
Since every node of $\omega^{<\omega}$ lies on a branch of $B$ and on
a branch of $B'$ (both sets meet every cylinder), $\tau$ is a
level-preserving bijection of $\omega^{<\omega}$ and each $\pi_t$ is
a permutation of $\omega$.

We construct $h$ by a back-and-forth recursion, as an increasing
union of finite partial $\Delta$-isometries, alternating between
extending the domain (forward steps, enumerating $B$) and the range
(backward steps, enumerating $B'$). Call a node \emph{treated} if it
is a prefix of a branch already in the domain of the current partial
map; the current partial map determines $\tau$ on treated nodes and
finitely many values of each $\pi_t$.

\emph{Base.} For the first branch $p_0\in B$ choose $h(p_0)\in B'$
arbitrarily. This determines $\tau$ along all prefixes of $p_0$ and,
at each node $t=p_0\restr j$, the single value
$\pi_t(p_0(j))=h(p_0)(j)$.

\emph{Forward step.} Let $p\in B$ be new, let
$m=\max\{\Delta(p,q):q\ \text{treated}\}$, realized by $q_*$, and let
$t=p\restr m$ (a treated node) and $n=p(m)$ (a fresh direction at
$t$: no treated branch passes through $t\concat n$). Any choice of a
new point
$h(p)\in B'\cap[\tau(t)\concat k]$ with $k$ different from the
finitely many directions already used at $\tau(t)$ extends the
$\Delta$-isometry: for treated $q$ with $\Delta(p,q)=m$ we get
$\Delta(h(p),h(q))=m$ because $k\neq\pi_t(q(m))$, and for
$\Delta(p,q)=j<m$ the ultrametric inequality gives
$\Delta(h(p),h(q))=\Delta(h(q_*),h(q))=j$, since
$h(p)\supseteq\tau(t)=h(q_*)\restr m$. Such a point exists because
$B'\cap[\tau(t)\concat k]$ is infinite and only finitely many points
are in use. This step assigns $\pi_t(n)=k$ (a \emph{controlled}
assignment, $k$ of our choosing among the fresh directions at
$\tau(t)$), and, at each node $u=p\restr j$ with $j>m$ (all fresh),
the \emph{block} assignments $\tau(u)=h(p)\restr j$ and
$\pi_{p\restr j}(p(j))=h(p)(j)$ dictated by the chosen point $h(p)$.

\emph{Backward step.} Symmetrically, for a new $p'\in B'$, with
separation node $\tau(t)$ and fresh direction $k$ there, we choose a
fresh direction $n$ at $t$ and a new point of $B\cap[t\concat n]$ to
map to $p'$; the assignment $\pi_t(n)=k$ is again controlled, and the
nodes of $p'$ beyond the separation level receive block assignments.

The union of the steps is a $\Delta$-preserving bijection $h:B\to B'$,
and every direction at every node is eventually treated, so each
$\pi_t$ is a total permutation. It remains to steer the controlled
assignments so that each $\pi_t$ has support at most $2$. Observe
that each node $t$ receives \emph{at most one} block assignment, at
its first visit: block assignments happen only at fresh nodes, which
are treated thereafter. Denote it $\pi_t(a)=b$ (at nodes first
visited by a controlled assignment there is no block pair, and the
argument below is simpler). We commit to the following rule for
controlled assignments at $t$: given a fresh direction $n$ (forward
case), set $\pi_t(n)=n$ if $n\notin\{a,b\}$, and $\pi_t(b)=a$ if
$n=b$. The remaining case $n=a$ cannot occur: the branch that created
the block assignment already passes through $t\concat a$, so this
direction is not fresh. We use the symmetric rule in the backward
case. The rule is always
available: by induction, at every stage the partial permutation
$\pi_t$ is contained in the union of the transposition $(a\,b)$ and
the identity, so for $n\notin\{a,b\}$ the value $n$ is not yet an
image (its only possible preimage under the rule is $n$ itself, and
$t\concat n$ was fresh), and for $n=b$ the value $a$ is not yet an
image (its only possible preimages are $b$, which was fresh, and $a$,
whose image is $b\ne a$); freshness on the $B'$ side is checked in
the same way. Hence every $\pi_t$ is contained in
$(a\,b)\cup\mathrm{id}$ and has support at most $2$.
\end{proof}

A $\Delta$-preserving bijection maps ultrametric balls onto balls,
that is, sets $B\cap[s]$ onto sets $B'\cap[t]$ with $|t|=|s|$; it
therefore respects the entire cell structure level by level.

\begin{proposition}[presentation independence]\label{prop:independence}
If $\langle A_s\rangle$ and $\langle A'_s\rangle$ are tree partitions
of $\omega$, there is a bijection $\sigma$ of $\omega$ carrying each
cell $A_s$ onto a cell $A'_t$ with $|t|=|s|$ and respecting the tree
structure; moreover $\sigma$ can be chosen so that the induced tree
isomorphism has all successor permutations of finite support.
Consequently, all ideals defined in this paper from the partition
(the hierarchies $\HH^{\mathcal J}_\alpha$ for every ideal
$\mathcal J\supseteq\Fin$, the $\GG_k$, $\Sc$, $\PC$, the $\ED_m$)
and the quotient $\Pow(\omega)/\HH$ do not depend on the chosen
partition, up to an isomorphism induced by a bijection of $\omega$.
\end{proposition}

\begin{proof}
Apply Lemma~\ref{lem:backforth} to the branch sets $B,B'$ of the two
partitions, and let $\sigma:\omega\to\omega$ be the point map
determined by $b'_{\sigma(x)}=h(b_x)$. Then
$\sigma[A_s]=A'_{\tau(s)}$ for every $s$: indeed $x\in A_s$ iff
$b_x\supseteq s$ iff $h(b_x)\supseteq\tau(s)$ iff
$\sigma(x)\in A'_{\tau(s)}$. Hence
$T'^\infty_{\sigma[X]}=\tau[T^\infty_X]$ and
$\Delta'(\sigma(x),\sigma(y))=\Delta(x,y)$ for all
$X\subseteq\omega$ and $x\neq y$.

Invariance now follows object by object. The chromatic ideals
$\GG_k$, defined from $\Delta$ alone, are transported by $\sigma$
directly. The ideals $\Sc$ and $\PC$ and the $\ED_m$ depend only on
the trace trees, the levels, and the cardinalities of successor sets,
all preserved by the level-preserving tree isomorphism $\tau$. For
the derived hierarchies, note that for every tree
$S\subseteq\omega^{<\omega}$ and node $t\in S$,
\[
 \suc_{\tau[S]}(\tau(t))=\pi_t[\suc_S(t)] .
\]
Since each $\pi_t$ has finite support, $\pi_t[E]$ differs from $E$ by
a finite set for every $E\subseteq\omega$; as
$\mathcal J\supseteq\Fin$, we get
$\pi_t[E]\in\mathcal J\iff E\in\mathcal J$. By induction on $\alpha$
(intersections at limits commute with images of bijections) this
yields
$D^\alpha_{\mathcal J}(\tau[S])=\tau[D^\alpha_{\mathcal J}(S)]$, and
since $\tau(\varnothing)=\varnothing$,
\[
 X\in\HH^{\mathcal J}_\alpha
 \iff
 \sigma[X]\in\HH'^{\mathcal J}_\alpha
\]
for every $\alpha<\omega_1$. The invariance of $\HH$ (the case
$\mathcal J=\Fin$, all $\alpha$ at once) gives the invariance of the
quotient $\Pow(\omega)/\HH$.
\end{proof}

\begin{corollary}[local homogeneity]\label{cor:homogeneity}
For every $s$, the localized partition
$\langle A_t:t\supseteq s\rangle$ of $A_s$ is a tree partition of the
countable set $A_s$ (reindex the cone by $u\mapsto A_{s\concat u}$);
hence every localized structure is isomorphic to the global one. In
particular $\HH(s)\cong\HH$ and $\mathbb B_s\cong\mathbb B$ (notation
of Section~\ref{sec:quotient}).
\end{corollary}

\begin{proof}
The reindexed family $\langle A_{s\concat u}:u\in\omega^{<\omega}
\rangle$ satisfies Definition~\ref{def:partition} over the countable
set $A_s$; transporting along any bijection $A_s\to\omega$ and
applying Proposition~\ref{prop:independence} identifies the localized
objects with the global ones.
\end{proof}

\subsection{The original tree ideals}

We recall the objects of \cite{Pelayo}, with the terminology of that
paper. A set $M\subseteq\omega$ is \emph{large} if it is infinite and
every node of $T^\infty_M$ has infinitely many children in
$T^\infty_M$; $X$ is \emph{dominating} if it contains a large set,
and $\HH$ is the ideal of non-dominating sets. A \emph{selector} is a
set contained in some $A_s$ with exactly one point in each
$A_{s\concat n}$; the ideal $\ED_m$ is generated by
$\{A_t:t\in\omega^{m+1}\}$ together with the selectors. The ideal
$\PC$ is generated by the selectors and the sets $X$ whose tree $T_X$
has slowly growing branching:
$|\suc_{T_X}(s)|\leq|s|+1$ for every $s\in T_X$.

\subsection{Kat\v etov orders and known facts}

For ideals $\mathcal I,\mathcal J$ on countable sets:
$\mathcal I\leq_K\mathcal J$ if there is
$f:\operatorname{dom}(\mathcal J)\to\operatorname{dom}(\mathcal I)$
with $A\in\mathcal I\Rightarrow f^{-1}[A]\in\mathcal J$; if $f$ can be
taken finite-to-one we write $\leq_{KB}$; $\sqsubseteq$ denotes a
bijective witness (an isomorphic copy). All ideals are assumed proper
and containing $\Fin$. An ideal $\mathcal I$ has property
$\mathrm{Kat}$ if $\mathcal I\leq_K\mathcal J$ implies
$\mathcal I\sqsubseteq\mathcal J$ for every $\mathcal J$.

For subsets $A\subseteq2^X$ and $C\subseteq2^Y$ of Cantor cubes we
write $A\leq_W C$ if there is a continuous map $F:2^X\to2^Y$ such
that $x\in A\iff F(x)\in C$.

Known facts and background used without proof:
\begin{enumerate}[label=(K\arabic*),leftmargin=3em]
 \item\label{K:rank} $\rk(\Fin^{\otimes r})=r$, where $\rk$ is the
 Borel separation rank, monotone under $\sqsubseteq$
 \cite{DebsSaintRaymond}.
 \item\label{K:kat} Each $\Fin^{\otimes r}$ has property
 $\mathrm{Kat}$ \cite[Example~4.1]{Barbarski}.
 \item\label{K:katFw} $\Fw$ has property $\mathrm{Kat}$
 \cite[Proposition~5.1]{Barbarski}.
 \item\label{K:kwela} There is an ideal $\Fwp$ of rank $\omega$ such
 that $\Fwp\leq_K\mathcal I\Leftrightarrow
 \forall r\,(\Fin^{\otimes r}\leq_K\mathcal I)$, where $\leq_K$ may be
 replaced by $\sqsubseteq$; $\Fwp$ and the canonical inductive limit
 $\Fw$ are not isomorphic, and indeed $\Fw$ has no isomorphic copy
 inside $\Fwp$ \cite[Theorems~4.6 and~5.1]{Kwela}.
 \item The Kat\v etov interval between $\ED$ and
 $\Fin\otimes\Fin$ contains copies of $\Pow(\omega)/\Fin$
 \cite{DasEtAl}.
 \item\label{K:colors} If $c$ is a coloring of pairs into finitely
 many colors, the ideal generated by its homogeneous sets, when proper, is $F_\sigma$ and
 tall; general theory in \cite{HrusakEtAl,CDU,LopezAbad}.
 \item Trace ideals of $\sigma$-ideals on Polish
 spaces and their quotients are developed in \cite{HrusakZapletal}.
 \item The transfinite hierarchy
 $\langle\mathrm{conv}_\alpha:\alpha<\omega_1\rangle$, critical for
 the countable compacta and strictly decreasing in the Kat\v etov
 order, is developed in \cite{FKK,FKK2}.
\end{enumerate}

\begin{remark}[scope relative to Kwela's conjecture]\label{rem:kwelascope}
Kwela's Conjecture~5.3 asks, for every nonzero countable limit ordinal
$\alpha$, for a rank-$\alpha$ ideal universal for the preceding
successor ideals, both in the Kat\v etov order and by isomorphic copies.
The case $\alpha=\omega$ is already established in \cite{Kwela} by the
ideal $\Fwp$. Thus the comparison of $\Fw$, $\Fwp$, and $\Hw$ in this
paper does not settle a new case of that conjecture. Rather, it studies
the non-uniqueness and coherence obstructions among three natural
amalgamations at the first limit level. The first ordinal not covered by
Kwela's theorem is $\omega\cdot2$.
\end{remark}

\section{The derivative parametrized by an auxiliary ideal}
\label{sec:DJ}

Let $\mathcal J\supseteq\Fin$ be a proper ideal on $\omega$. For
$T\subseteq\omega^{<\omega}$ and $s\in\omega^{<\omega}$ let
$\suc_T(s)=\{n:s\concat n\in T\}$.

\begin{definition}\label{def:derivative}
$D_{\mathcal J}(T)=\{s\in T:\suc_T(s)\notin\mathcal J\}$, iterated as
$D^{\alpha}_{\mathcal J}$ (intersections at limits). For
$\alpha<\omega_1$,
\[
 \HH^{\mathcal J}_\alpha=\{X:\varnothing\notin
   D^\alpha_{\mathcal J}(T^\infty_X)\},\qquad
 \HH^{\mathcal J}=\bigcup_{\alpha<\omega_1}\HH^{\mathcal J}_\alpha.
\]
A tree $L\ni\varnothing$ is \emph{$\mathcal J$-positive Laver} if
$\suc_L(s)\notin\mathcal J$ for all $s\in L$. We write
$\HH_\alpha=\HH^{\Fin}_\alpha$ and $\HH=\HH^{\Fin}$; by
Theorem~\ref{thm:nucleus} and Lemma~\ref{lem:skeleton}, this agrees
with the objects of \cite{Pelayo}.
\end{definition}

Iterated derivatives of a tree need not be downward closed, so we
record three elementary facts that will be used repeatedly.

\begin{lemma}\label{lem:persistence}
Let $S,T\subseteq\omega^{<\omega}$ and $\alpha<\omega_1$.
\begin{enumerate}[label=(\alph*)]
 \item (monotonicity) If $S\subseteq T$ then
 $D^\alpha_{\mathcal J}(S)\subseteq D^\alpha_{\mathcal J}(T)$.
 \item (locality) For $t\in\omega^{<\omega}$, whether
 $t\in D^\alpha_{\mathcal J}(S)$ depends only on
 $S\cap\{u:u\supseteq t\}$.
 \item (persistence) If $S$ is a tree and
 $\suc_{D^\alpha_{\mathcal J}(S)}(s)\notin\mathcal J$, then
 $s\in D^\alpha_{\mathcal J}(S)$.
\end{enumerate}
\end{lemma}

\begin{proof}
(a) and (b) are immediate inductions on $\alpha$: the successor step
holds because $\suc$ is monotone and looks only above $t$, and limits
are intersections.

(c) By induction on $\alpha$. For $\alpha=0$: the hypothesis gives
some $s\concat n\in S$, and $S$ is a tree, so $s\in S$. Successor:
since
$D^{\alpha+1}_{\mathcal J}(S)\subseteq D^\alpha_{\mathcal J}(S)$, the
hypothesis gives
$\suc_{D^\alpha_{\mathcal J}(S)}(s)\supseteq
\suc_{D^{\alpha+1}_{\mathcal J}(S)}(s)\notin\mathcal J$, so
$s\in D^\alpha_{\mathcal J}(S)$ by the induction hypothesis, and then
$s\in D_{\mathcal J}(D^\alpha_{\mathcal J}(S))
=D^{\alpha+1}_{\mathcal J}(S)$ by definition. Limit $\lambda$: for
each $\beta<\lambda$,
$\suc_{D^\beta_{\mathcal J}(S)}(s)\supseteq
\suc_{D^\lambda_{\mathcal J}(S)}(s)\notin\mathcal J$, so
$s\in D^\beta_{\mathcal J}(S)$ for all $\beta<\lambda$, i.e.\
$s\in D^\lambda_{\mathcal J}(S)$.
\end{proof}

\begin{theorem}[nucleus]\label{thm:nucleus}
The following are equivalent: (1) $X\notin\HH^{\mathcal J}$; (2)
$T^\infty_X$ contains a $\mathcal J$-positive Laver tree; (3) the root
survives all countable derivatives.
\end{theorem}

\begin{proof}
The derivatives form a decreasing sequence of subsets of the
countable set $\omega^{<\omega}$, so there is $\alpha_*<\omega_1$
with
$D^{\alpha_*}_{\mathcal J}(T^\infty_X)
=D^{\alpha_*+1}_{\mathcal J}(T^\infty_X)=:K$, and then
$D^\beta_{\mathcal J}(T^\infty_X)=K$ for all $\beta\geq\alpha_*$.
(1)$\iff$(3) is immediate from this stabilization. Suppose that $\varnothing\in K$, and take the rooted component
\[
 L=\{s\in K:\text{ every initial segment of }s\text{ belongs to }K\}.
\]
Then $L$ is a tree containing the root. If $s\in L$ and
$n\in\suc_K(s)$, every initial segment of $s\concat n$ belongs to
$K$, so $s\concat n\in L$. Hence
$\suc_L(s)=\suc_K(s)\notin\mathcal J$ for every $s\in L$, because
$D_{\mathcal J}(K)=K$. Thus $L$ is a $\mathcal J$-positive Laver tree
contained in $T^\infty_X$; this gives (3)$\Rightarrow$(2).
Conversely, a $\mathcal J$-positive Laver tree $L$ satisfies
$D_{\mathcal J}(L)=L$, hence, by Lemma~\ref{lem:persistence}(a) and
induction on $\alpha$,
$L=D^\alpha_{\mathcal J}(L)\subseteq
D^\alpha_{\mathcal J}(T^\infty_X)$ for all $\alpha$, so the root
survives: (2)$\Rightarrow$(3).
\end{proof}

\begin{theorem}[strict hierarchy]\label{thm:hierarchy}
Each $\HH^{\mathcal J}_\alpha$ is a proper ideal, tall for
$\alpha\geq1$, and
$\HH^{\mathcal J}_\alpha\subsetneq\HH^{\mathcal J}_{\alpha+1}$ for all
$\alpha<\omega_1$.
\end{theorem}

\begin{proof}
\emph{Ideal.} Downward closure follows from
Lemma~\ref{lem:persistence}(a), since $X\subseteq Y$ implies
$T^\infty_X\subseteq T^\infty_Y$. For unions it suffices to prove,
for trees $S,T$,
\[
 D^\alpha_{\mathcal J}(S\cup T)\subseteq
 D^\alpha_{\mathcal J}(S)\cup D^\alpha_{\mathcal J}(T),
\]
by induction on $\alpha$, and to apply it to
$T^\infty_{X\cup Y}=T^\infty_X\cup T^\infty_Y$. The case $\alpha=0$
is trivial. Successor: let
$s\in D^{\alpha+1}_{\mathcal J}(S\cup T)$; then
$\suc_{D^\alpha_{\mathcal J}(S\cup T)}(s)\notin\mathcal J$, and by
the induction hypothesis this set is contained in
$\suc_{D^\alpha_{\mathcal J}(S)}(s)\cup
\suc_{D^\alpha_{\mathcal J}(T)}(s)$, so one of the two is not in
$\mathcal J$, say the one for $S$. By
Lemma~\ref{lem:persistence}(c), $s\in D^\alpha_{\mathcal J}(S)$, and
then $s\in D^{\alpha+1}_{\mathcal J}(S)$. Limit $\lambda$: the
families $A_\beta=D^\beta_{\mathcal J}(S)$ and
$B_\beta=D^\beta_{\mathcal J}(T)$ are decreasing, and for decreasing
families
$\bigcap_{\beta<\lambda}(A_\beta\cup B_\beta)
=\bigcap_\beta A_\beta\cup\bigcap_\beta B_\beta$: if $x$ avoids some
$A_{\beta_0}$ and some $B_{\beta_1}$, it avoids
$A_\beta\cup B_\beta$ for $\beta=\max(\beta_0,\beta_1)$. The
induction hypothesis at each $\beta<\lambda$ now gives
$D^\lambda_{\mathcal J}(S\cup T)\subseteq
\bigcap_\beta(A_\beta\cup B_\beta)
=D^\lambda_{\mathcal J}(S)\cup D^\lambda_{\mathcal J}(T)$.

\emph{Proper.} $T^\infty_\omega=\omega^{<\omega}$ is a
$\mathcal J$-positive Laver tree ($\mathcal J$ is proper), so the
root survives all derivatives and
$\omega\notin\HH^{\mathcal J}_\alpha$.

\emph{Tall for $\alpha\geq1$.} Let $X$ be infinite. If some section
$X\cap A_{\langle i\rangle}$ is infinite, let $Y$ be an infinite
subset of it; then
$\suc_{T^\infty_Y}(\varnothing)\subseteq\{i\}\in\mathcal J$, so the
root dies at the first derivative and
$Y\in\HH^{\mathcal J}_1\subseteq\HH^{\mathcal J}_\alpha$. If all
first-level sections of $X$ are finite, then
$\suc_{T^\infty_X}(\varnothing)=\varnothing\in\mathcal J$ and $X$
itself is in $\HH^{\mathcal J}_1$.

\emph{Strictness.} By recursion on $\alpha$ we construct trees
$S_\alpha\subseteq\omega^{<\omega}$ whose root \emph{dies exactly at
stage $\alpha+1$}:
$\varnothing\in D^\alpha_{\mathcal J}(S_\alpha)\setminus
D^{\alpha+1}_{\mathcal J}(S_\alpha)$. Put $S_0=\{\varnothing\}$: the
root has empty successor set and dies at stage $1$. Successor:
$S_{\alpha+1}=\{\varnothing\}\cup\bigcup_{n\in\omega}
\langle n\rangle\concat S_\alpha$, a copy of $S_\alpha$ below every
child of the root. By Lemma~\ref{lem:persistence}(b), each copy
evolves under the derivative exactly as $S_\alpha$ does, so
$\langle n\rangle\in D^\beta_{\mathcal J}(S_{\alpha+1})$ iff
$\beta\leq\alpha$. Hence for every $\beta\leq\alpha$ the root has
successor set $\omega\notin\mathcal J$ inside
$D^\beta_{\mathcal J}(S_{\alpha+1})$, so it belongs to
$D^{\beta+1}_{\mathcal J}(S_{\alpha+1})$ by
Lemma~\ref{lem:persistence}(c); in particular it survives to stage
$\alpha+1$. At stage $\alpha+2$ its surviving successor set is
$\varnothing\in\mathcal J$: the root dies exactly at
$(\alpha+1)+1$. Limit $\lambda$: fix $\alpha_n\uparrow\lambda$ and
put $S_\lambda=\{\varnothing\}\cup\bigcup_n\langle n\rangle
\concat S_{\alpha_n}$. For every $\beta<\lambda$ the set of children
alive at stage $\beta$ is $\{n:\alpha_n\geq\beta\}$, which is
cofinite, hence not in $\mathcal J$ ($\mathcal J$ is proper), so the
root survives every stage $\beta<\lambda$, and also stage $\lambda$
(an intersection); at stage $\lambda+1$ the surviving successor set
is $\{n:\alpha_n\geq\lambda\}=\varnothing\in\mathcal J$ and the root
dies.

Realize $X_\alpha=X_{S_\alpha}$ by Lemma~\ref{lem:realization}; then
$T^\infty_{X_\alpha}=S_\alpha$ and
$X_\alpha\in\HH^{\mathcal J}_{\alpha+1}\setminus
\HH^{\mathcal J}_\alpha$, proving
$\HH^{\mathcal J}_\alpha\subsetneq\HH^{\mathcal J}_{\alpha+1}$.
\end{proof}

\begin{theorem}[Fubini formula]\label{thm:fubini}
For every node $s$ and $\alpha<\omega_1$, with the localized ideals of
the partition of $A_s$,
\[
 \HH^{\mathcal J}_{\alpha+1}(s)
 =\Bigl\{X\subseteq A_s:
   \{n:X\cap A_{s\concat n}\notin\HH^{\mathcal J}_\alpha(s\concat n)\}
   \in\mathcal J\Bigr\},
\]
and at limits the hierarchy is the union of the previous levels:
$\HH^{\mathcal J}_\lambda(s)
=\bigcup_{\beta<\lambda}\HH^{\mathcal J}_\beta(s)$.
Consequently, for $m<\omega$,
\[
 \HH^{\mathcal J}_m\cong
 \underbrace{\mathcal J\otimes\cdots\otimes\mathcal J}_{m}\otimes\Fin;
 \qquad
 \HH^{\Fin^{\otimes r}}_m\cong\Fin^{\otimes(rm+1)};
 \qquad
 \HH_n\cong\Fin^{\otimes(n+1)}.
\]
\end{theorem}

\begin{proof}
For $X\subseteq A_s$ write $S=T^\infty_X\cap\{u:u\supseteq s\}$. By
Lemma~\ref{lem:persistence}(b), localization and derivation commute:
membership of $X$ in $\HH^{\mathcal J}_\alpha(s)$ means exactly that
$s$ dies within $\alpha$ derivatives of $S$, because nodes outside
the cone of $s$ are irrelevant to what happens at and above $s$.

\emph{Successor formula.} By locality once more, for each $n$,
\[
 s\concat n\in D^\alpha_{\mathcal J}(S)
 \iff
 X\cap A_{s\concat n}\notin\HH^{\mathcal J}_\alpha(s\concat n).
\]
Write $N=\{n:s\concat n\in D^\alpha_{\mathcal J}(S)\}$. If
$N\notin\mathcal J$, then $s\in D^{\alpha+1}_{\mathcal J}(S)$ by
Lemma~\ref{lem:persistence}(c). If $N\in\mathcal J$, then either
$s\notin D^\alpha_{\mathcal J}(S)$, or
$\suc_{D^\alpha_{\mathcal J}(S)}(s)=N\in\mathcal J$; in both cases
$s\notin D^{\alpha+1}_{\mathcal J}(S)$. This is the displayed
equivalence.

\emph{Limits.} $s\notin D^\lambda_{\mathcal J}(S)
=\bigcap_{\beta<\lambda}D^\beta_{\mathcal J}(S)$ iff
$s\notin D^\beta_{\mathcal J}(S)$ for some $\beta<\lambda$.

\emph{Isomorphisms.} The first is proved by induction on $m$. For
$m=0$: $X\in\HH^{\mathcal J}_0(s)$ iff $s\notin T^\infty_X$ iff
$X\cap A_s$ is finite, so $\HH^{\mathcal J}_0(s)=\Fin(A_s)$.
The successor formula says precisely that
$\HH^{\mathcal J}_{m+1}(s)$ is the $\mathcal J$-Fubini sum of the
ideals $\HH^{\mathcal J}_m(s\concat n)$; by the induction hypothesis
each summand is isomorphic to $\mathcal J^{\otimes m}\otimes\Fin$,
and assembling bijections cell by cell gives
$\HH^{\mathcal J}_{m+1}(s)\cong\mathcal J\otimes(\mathcal J^{\otimes
m}\otimes\Fin)=\mathcal J^{\otimes(m+1)}\otimes\Fin$. Concretely, for
$s=\varnothing$ the isomorphism reads the first $m$ coordinates of
the branch and an enumeration of the level-$m$ cell:
$x\mapsto(b_x(0),\dots,b_x(m-1),e_{b_x\restr m}(x))$, where
$e_u:A_u\to\omega$ are fixed bijections. The second isomorphism
follows by associativity of the Fubini product:
$(\Fin^{\otimes r})^{\otimes m}\otimes\Fin\cong\Fin^{\otimes(rm+1)}$.
The third is the case $r=1$.
\end{proof}

\begin{corollary}\label{cor:chain}
$\HH_n<_{K}\HH_{n+1}$ and $\HH_n<_{KB}\HH_{n+1}$ for all $n$; the
identity witnesses the upward direction. The same holds for
$\HH^{\Fin^{\otimes r}}_m$ with jumps of $r$ powers.
\end{corollary}

\begin{proof}
$\HH_n\subseteq\HH_{n+1}$, so the identity (which is finite-to-one)
witnesses $\HH_n\leq_{KB}\HH_{n+1}$. If conversely
$\HH_{n+1}\leq_K\HH_n$, then, via the isomorphisms of
Theorem~\ref{thm:fubini},
$\Fin^{\otimes(n+2)}\leq_K\Fin^{\otimes(n+1)}$; by \ref{K:kat} there
would be an isomorphic copy
$\Fin^{\otimes(n+2)}\sqsubseteq\Fin^{\otimes(n+1)}$, and \ref{K:rank}
would give $n+2\leq n+1$, a contradiction. The same argument applies
to $\HH^{\Fin^{\otimes r}}_m\cong\Fin^{\otimes(rm+1)}$.
\end{proof}

\begin{proposition}\label{prop:ED}
For all $m$ and $n\geq1$: $\ED_m\subseteq\HH_1$ (the identity is a
$\leq_{KB}$-witness) and $\HH_1\not\leq_K\ED_m$; hence
$\ED_m<_{KB}\HH_n$.
\end{proposition}

\begin{proof}
Each generator of $\ED_m$ (a cell of level $\geq1$, or a selector)
has at most one infinite first-level section, so the root of its
trace tree has at most one child and dies at the first derivative:
the generators, and hence all of $\ED_m$, lie in the ideal $\HH_1$.
If $\HH_1\leq_K\ED_m$, then $\Fin^{\otimes2}\leq_K\ED_m$
(Theorem~\ref{thm:fubini}); by \ref{K:kat} there would be a copy
$\Fin^{\otimes2}\sqsubseteq\ED_m$, but $\ED_m$ is $F_\sigma$
\cite{Pelayo}, hence of separation rank $\leq1$, while
$\rk(\Fin^{\otimes2})=2$, contradicting \ref{K:rank}. Finally
$\ED_m\leq_{KB}\HH_1\leq_{KB}\HH_n$, while $\HH_n\leq_K\ED_m$ would
give $\HH_1\leq_K\ED_m$ by composing with the inclusion.
\end{proof}

\begin{proposition}\label{prop:nestedlimit}
$\Hw:=\bigcup_n\HH_n=\HH_\omega$ is a tall ideal of class
$\boldsymbol\Sigma^0_\omega$, of separation rank exactly $\omega$,
strictly above every $\HH_n$ in $\leq_K$ and $\leq_{KB}$, and
$\Fwp\sqsubseteq\Hw$.
\end{proposition}

\begin{proof}
$\Hw=\HH_\omega$ is the limit clause of Theorem~\ref{thm:fubini}; it
is an ideal (an increasing union of ideals) and tall (it contains the
tall ideal $\HH_1$). Each $\HH_n\cong\Fin^{\otimes(n+1)}$ is of
additive Borel class $\boldsymbol\Sigma^0_{2n+2}$, so the countable
union $\Hw$ is of class $\boldsymbol\Sigma^0_\omega$. Since the
proper ideal $\Hw$ is disjoint from its dual filter, $\Hw$ itself is
a $\boldsymbol\Sigma^0_\omega$ separator; the definition of the
separation rank therefore gives $\rk(\Hw)\leq\omega$. For the lower
bound: for every $n$,
the isomorphism $\Fin^{\otimes(n+1)}\cong\HH_n$ composed with the
inclusion $\HH_n\subseteq\Hw$ is a bijective Kat\v etov witness,
i.e.\ $\Fin^{\otimes(n+1)}\sqsubseteq\Hw$, and monotonicity of the
rank under copies (\ref{K:rank}) gives $\rk(\Hw)\geq n+1$ for all
$n$.

Strictness: the identity witnesses $\HH_n\leq_{KB}\Hw$. If
$\Hw\leq_K\HH_n$ for some $n$, then composing with the inclusion
witness of $\HH_{n+1}\leq_K\Hw$ would give
$\HH_{n+1}\leq_K\HH_n$, contradicting Corollary~\ref{cor:chain}.

Finally, $\Fin^{\otimes r}\leq_K\Hw$ for every $r$ (through
$\HH_{r-1}$), so \ref{K:kwela} yields $\Fwp\sqsubseteq\Hw$.
\end{proof}

\subsection{Descriptive complexity}

\begin{theorem}\label{thm:complexity}
If $\mathcal J$ is Borel (resp.\ analytic), each
$\HH^{\mathcal J}_\alpha$ is Borel (resp.\ analytic). Moreover
$\mathcal J\leq_W\HH^{\mathcal J}_1$ via
$B\mapsto\bigcup_{n\in B}A_{\langle n\rangle}$.
\end{theorem}

\begin{proof}
By induction on $\alpha$ we show that, for every node $s$, the set
$\mathcal H(\alpha,s)=\{X\subseteq\omega:
X\cap A_s\in\HH^{\mathcal J}_\alpha(s)\}$ is Borel (resp.\
analytic); the theorem is the case $s=\varnothing$. For $\alpha=0$,
$\mathcal H(0,s)=\{X:X\cap A_s\ \text{finite}\}$ is
$\boldsymbol\Sigma^0_2$. At limits, by Theorem~\ref{thm:fubini},
$\mathcal H(\lambda,s)=\bigcup_{\beta<\lambda}\mathcal H(\beta,s)$ is
a countable union.

Successor, $\mathcal J$ Borel: by Theorem~\ref{thm:fubini},
\[
 X\in\mathcal H(\alpha+1,s)
 \iff
 \sigma_s(X):=\{n:X\notin\mathcal H(\alpha,s\concat n)\}\in\mathcal J.
\]
Each bit of $\sigma_s(X)$ is a Borel condition on $X$ by the
induction hypothesis, so
$\sigma_s:\Pow(\omega)\to\Pow(\omega)$ is a Borel map and
$\mathcal H(\alpha+1,s)=\sigma_s^{-1}[\mathcal J]$ is Borel.

Successor, $\mathcal J$ analytic: rewrite the condition in witness
form,
\[
 X\in\mathcal H(\alpha+1,s)
 \iff
 \exists A\in\mathcal J\ \forall n\notin A\quad
 X\in\mathcal H(\alpha,s\concat n).
\]
For a fixed witness the inner condition is a countable intersection
of analytic sets, hence analytic, and the whole set is the projection
of the analytic set
$\{(A,X):A\in\mathcal J\ \wedge\ \forall n\notin A\
(X\in\mathcal H(\alpha,s\concat n))\}$, hence analytic.

The map $B\mapsto X_B:=\bigcup_{n\in B}A_{\langle n\rangle}$ is
continuous (each point of $\omega$ lies in exactly one first-level
cell, so each bit of $X_B$ reads one bit of $B$). Its trace tree is
$T^\infty_{X_B}=\{\varnothing\}\cup\{\langle n\rangle\concat t:
n\in B,\ t\in\omega^{<\omega}\}$: a full cone below each $n\in B$,
and full cones are $\mathcal J$-positive Laver. Hence the surviving
successor set of the root after one derivative is exactly $B$, and
$X_B\in\HH^{\mathcal J}_1$ iff $B\in\mathcal J$: an exact Wadge
reduction.
\end{proof}

\begin{theorem}\label{thm:Pi11}
If $\mathcal J$ is a proper Borel ideal with
$\Fin\subseteq\mathcal J$, then $\HH^{\mathcal J}$ is
$\mathbf\Pi^1_1$-complete.
\end{theorem}

\begin{proof}
\emph{Upper bound.} By Theorem~\ref{thm:nucleus},
$X\notin\HH^{\mathcal J}$ iff there exists
$L\subseteq\omega^{<\omega}$ which is a $\mathcal J$-positive Laver
tree with $L\subseteq T^\infty_X$. For Borel $\mathcal J$ the matrix
of this condition is Borel in the pair $(X,L)$ (a countable
conjunction of the conditions ``$\suc_L(s)\notin\mathcal J$'' and
``$s\in L\Rightarrow X\cap A_s$ infinite''), so the complement of
$\HH^{\mathcal J}$ is analytic.

\emph{Splitting.} Every $\mathcal J$-positive set $P\subseteq\omega$
splits into two $\mathcal J$-positive sets. Otherwise, for every
partition $P=P_0\cup P_1$ into two pieces exactly one piece is
positive, so $\{A\subseteq P:A\notin\mathcal J\}$ is an ultrafilter
on $P$; it is nonprincipal (as $\Fin\subseteq\mathcal J$) and Borel,
hence has the Baire property. But it is invariant under finite
symmetric differences, so the topological zero--one law makes it
meager or comeager; complementation maps it onto its complement,
making either alternative impossible (the standard category proof of
Sierpi\'nski's theorem; see \cite[Section~8]{Kechris}). Iterating, we obtain pairwise
disjoint $\mathcal J$-positive sets
$\langle B_a:a\in\omega\rangle$: split $\omega=P_0\cup B_0$ with both
pieces positive, then $P_0=P_1\cup B_1$, and so on.

\emph{Fattening.} For a tree $R\subseteq\omega^{<\omega}$ define
\[
 \Fat(R)=\{t\in\omega^{<\omega}:
 \exists r\in R\ \ |r|=|t|\ \wedge\
 \forall i<|t|\ \ t(i)\in B_{r(i)}\}.
\]
Since the $B_a$ are pairwise disjoint, the witness $r$ is unique
(each $t(i)$ lies in at most one $B_a$), so there is a well-defined
reading map $t\mapsto r_t$ on $\Fat(R)$, coherent under restriction,
and $\Fat(R)$ is a tree. We claim: $R$ is ill-founded iff $\Fat(R)$
contains a $\mathcal J$-positive Laver tree. If $z$ is an infinite
branch of $R$, then
$L_z=\{t:\forall i<|t|\ \ t(i)\in B_{z(i)}\}\subseteq\Fat(R)$ is
$\mathcal J$-positive Laver, since
$\suc_{L_z}(t)=B_{z(|t|)}\notin\mathcal J$. Conversely, if
$L\subseteq\Fat(R)$ is $\mathcal J$-positive Laver, then $L$ has an
infinite branch (successor sets are nonempty), and the reading map
carries it to an infinite branch of $R$.

\emph{Reduction.} Compose with the Realization Lemma: the map
$R\mapsto X_{\Fat(R)}$ is continuous (whether a fixed point $p_{s,n}$
belongs to $X_{\Fat(R)}$ asks whether $s\in\Fat(R)$, which reads the
single bit ``$r_s\in R$'' when the candidate witness $r_s$ exists,
and is constantly false otherwise), and by
Theorem~\ref{thm:nucleus} together with the claim,
\[
 R\ \text{well-founded}
 \iff
 X_{\Fat(R)}\in\HH^{\mathcal J}.
\]
This reduces the $\mathbf\Pi^1_1$-complete set of well-founded trees
continuously to $\HH^{\mathcal J}$.
\end{proof}

\subsection{Functoriality}

For $A\subseteq\omega^{<\omega}$ write
$\mathord\downarrow A=\{t:\exists s\in A\ (t\subseteq s)\}$. For
$1\leq\alpha<\omega_1$ define genuine ideals on the set of nodes by
\[
 \mathfrak R^{\mathcal J}_\alpha
 =\{A\subseteq\omega^{<\omega}:
   \varnothing\notin D^\alpha_{\mathcal J}(\mathord\downarrow A)\},
 \qquad
 \mathfrak R^{\mathcal J}
 =\bigcup_{1\leq\alpha<\omega_1}\mathfrak R^{\mathcal J}_\alpha .
\]
Indeed, downward closure preserves inclusions and finite unions, so
the union argument in the proof of Theorem~\ref{thm:hierarchy}
applies. Finite sets die after the first derivative, the full tree
survives every derivative, and the hierarchy is increasing; hence
these are proper ideals containing all finite sets.

\begin{theorem}\label{thm:functorial}
If $f$ witnesses $\mathcal J\leq_K\mathcal K$, then its coordinatewise
action $f^{<\omega}$ witnesses
$\mathfrak R^{\mathcal J}_\alpha\leq_K\mathfrak R^{\mathcal K}_\alpha$
for every $1\leq\alpha<\omega_1$, and
$\mathfrak R^{\mathcal J}\leq_K\mathfrak R^{\mathcal K}$;
finite-to-one is preserved.
\end{theorem}

\begin{proof}
We first record the tree-level statement. By induction on $\alpha$
one proves the inclusion
\[
 D_{\mathcal K}^{\alpha}
 \bigl((f^{<\omega})^{-1}[S]\bigr)
 \subseteq
 (f^{<\omega})^{-1}
 \bigl[D_{\mathcal J}^{\alpha}(S)\bigr].
\]
Successor step: let
$t\in D^{\alpha+1}_{\mathcal K}((f^{<\omega})^{-1}[S])$, so that
$E=\suc_{D^\alpha_{\mathcal K}((f^{<\omega})^{-1}[S])}(t)
\notin\mathcal K$. By the induction hypothesis,
$E\subseteq f^{-1}\bigl[\suc_{D^\alpha_{\mathcal J}(S)}
(f^{<\omega}(t))\bigr]$; if
$\suc_{D^\alpha_{\mathcal J}(S)}(f^{<\omega}(t))$ were in
$\mathcal J$, its $f$-preimage would be in $\mathcal K$ and would
contain $E$, a contradiction. Hence, by
Lemma~\ref{lem:persistence}(c) on the $\mathcal J$-side,
$f^{<\omega}(t)\in D^{\alpha+1}_{\mathcal J}(S)$. Limits are intersections.

Now let $A\in\mathfrak R^{\mathcal J}_\alpha$ and put
$S=\mathord\downarrow A$. Evaluating the displayed inclusion at the
root shows that $(f^{<\omega})^{-1}[S]$ is small on the
$\mathcal K$-side. Moreover,
\[
 \mathord\downarrow\bigl((f^{<\omega})^{-1}[A]\bigr)
 \subseteq (f^{<\omega})^{-1}[\mathord\downarrow A].
\]
Monotonicity therefore gives
$(f^{<\omega})^{-1}[A]\in\mathfrak R^{\mathcal K}_\alpha$. This proves
the reduction at every fixed $\alpha$; choosing an $\alpha$ that
contains a given member proves the reduction for the unions. Finally,
if $f$ is finite-to-one, so is the length-preserving map
$f^{<\omega}$.
\end{proof}

\begin{proposition}\label{prop:lifting}
If in addition there is a finite-to-one $g:\omega\to\omega$ with
$b_{g(x)}=f\circ b_x$, then $g$ witnesses
$\HH^{\mathcal J}_\alpha\leq_{KB}\HH^{\mathcal K}_\alpha$ and
$\HH^{\mathcal J}\leq_{KB}\HH^{\mathcal K}$.
\end{proposition}

\begin{proof}
We first check that
\[
 T^\infty_{g^{-1}[X]}
 \subseteq
 (f^{<\omega})^{-1}[T^\infty_X]
 \qquad(X\subseteq\omega).
\]
Let $t\in T^\infty_{g^{-1}[X]}$, i.e.\ infinitely many
$x\in g^{-1}[X]$ have $b_x\supseteq t$. Since $g$ is finite-to-one,
the set of their images $g(x)\in X$ is infinite, and each satisfies
$b_{g(x)}=f\circ b_x\supseteq f^{<\omega}(t)$; hence
$X\cap A_{f^{<\omega}(t)}$ is infinite and
$f^{<\omega}(t)\in T^\infty_X$.

Now let $X\in\HH^{\mathcal J}_\alpha$, so
$\varnothing\notin D^\alpha_{\mathcal J}(T^\infty_X)$. By
Lemma~\ref{lem:persistence}(a) and the tree-level inclusion proved in
Theorem~\ref{thm:functorial},
\[
 D^\alpha_{\mathcal K}(T^\infty_{g^{-1}[X]})
 \subseteq
 D^\alpha_{\mathcal K}\bigl((f^{<\omega})^{-1}[T^\infty_X]\bigr)
 \subseteq
 (f^{<\omega})^{-1}\bigl[D^\alpha_{\mathcal J}(T^\infty_X)\bigr],
\]
and the last set does not contain $\varnothing$. Hence
$g^{-1}[X]\in\HH^{\mathcal K}_\alpha$, and $g$ is finite-to-one.
\end{proof}

\begin{remark}
The lifting hypothesis is not implied by
$\mathcal J\leq_K\mathcal K$: it asks the branch set to be closed
under the coordinatewise map. Beyond inclusion
($\mathcal J\subseteq\mathcal K$ gives
$\HH^{\mathcal J}\subseteq\HH^{\mathcal K}$, with the identity as
witness) and explicit liftings, the functoriality of
$\HH^{\mathcal J}$ at the level of points remains open; see
Problem~\ref{prob:DJ}.
\end{remark}

\begin{remark}[the binary case, in one line]\label{rem:binary}
Replacing $\omega^{<\omega}$ by $2^{<\omega}$ does not produce a
parallel theory: the condition ``infinitely many successors'' is
vacuous, requiring both children does not give an ideal, and the
binary derivative stabilizes at $\omega$ with core the classical
nowhere dense ideal of the countable dense set
\cite{FarahSolecki,KwelaSabok}. The correct transfinite counterpart
on the compact space is the hierarchy $\mathrm{conv}_\alpha$ of
\cite{FKK,FKK2}, already developed and strictly decreasing in the
Kat\v etov order.
\end{remark}

\section{The main theorem: \texorpdfstring{$\Fw\not\leq_K\Hw$}{Fin-omega is not Katetov below the nested limit}}\label{sec:main}

Throughout this section $\Fin^{\otimes k}$ has the \emph{first}
coordinate outermost: $C\in\Fin^{\otimes(k+1)}$ iff all but finitely
many sections $C_{n}=\{t:(n)\concat t\in C\}$ belong to
$\Fin^{\otimes k}$. For $k\leq N$,
$\pi_{k,N}:\omega^N\to\omega^k$ denotes the projection onto the
\emph{last} $k$ coordinates. We write $\Fw$ for the canonical
inductive limit of the finite Fubini powers along these projections,
as defined by Kwela
\cite[Definition~2.1 and Subsection~2.4]{Kwela}. For the system of
suffix projections his definition unwinds to the following
combinatorial form, which is the one we use: the domain of $\Fw$ is
the disjoint sum $D=\bigsqcup_{m\geq1}\omega^m$, and for
$M\subseteq D$,
\begin{equation}\label{eq:Fw}
 M\in\Fw
 \iff
 \exists k\geq1\ \exists B\in\Fin^{\otimes k}\
 \forall m\geq k\quad
 \pi_{k,m}[M\cap\omega^m]\subseteq B.
\end{equation}
Indeed, Kwela's Definition~2.1 says that $M$ is small precisely when
there are $k\geq1$ and
$P\in(\Fin^{\otimes k})^\star$ such that
\[
 M\cap\bigsqcup_{m\geq k}\pi_{k,m}^{-1}[P]=\varnothing .
\]
Putting $B=\omega^k\setminus P\in\Fin^{\otimes k}$ turns this
condition into
$\pi_{k,m}[M\cap\omega^m]\subseteq B$ for every $m\geq k$, and the
same calculation backwards gives the converse. Thus
\eqref{eq:Fw} is exactly Kwela's inductive limit for the full-domain
suffix maps, not merely an isomorphic presentation.

We also fix an elementary insertion lemma for cylinders over
Fubini-small sets.

\begin{lemma}[insertion]\label{lem:insertion}
Let $1\leq j_1<\dots<j_r\leq k$ and let
$\pi:\omega^k\to\omega^r$, $\pi(t)=(t(j_1),\dots,t(j_r))$, be the
corresponding ordered projection (coordinates counted from the
outermost). If $B\in\Fin^{\otimes r}$ then
$\pi^{-1}[B]\in\Fin^{\otimes k}$.
\end{lemma}

\begin{proof}
Induction on $k$. For $k=1$ we have $r=1$ and $\pi$ the identity.
Let $k>1$. If $j_1=1$ and $r=1$, then
$\pi^{-1}[B]=B\times\omega^{k-1}$ with $B\in\Fin$; only finitely many
outer sections are nonempty, so $\pi^{-1}[B]\in\Fin^{\otimes k}$.
Assume now that $j_1=1$ and $r\geq2$. The section of
$\pi^{-1}[B]$ at first coordinate $n$ is
$\pi'^{-1}[B_n]$, where $\pi'$ is the ordered projection of
$\omega^{k-1}$ given by $j_2-1<\dots<j_r-1$ and $B_n$ is the
corresponding section of $B$. All but finitely many $B_n$ lie in
$\Fin^{\otimes(r-1)}$, and for those the induction hypothesis puts
the section in $\Fin^{\otimes(k-1)}$. If $j_1>1$ (the outermost
coordinate is free), every section equals $\pi'^{-1}[B]$ for the
shifted projection $\pi'$ of $\omega^{k-1}$, which lies in
$\Fin^{\otimes(k-1)}$ by the induction hypothesis. In both cases
$\pi^{-1}[B]\in\Fin\otimes\Fin^{\otimes(k-1)}=\Fin^{\otimes k}$.
\end{proof}

\subsection{The coordinate criterion and finite fragments}

\begin{lemma}[tower criterion]\label{lem:tower}
For an ideal $\mathcal I$ on a countable set $\Omega$ the following
are equivalent:
\begin{enumerate}[label=(\alph*)]
 \item $\Fw\leq_K\mathcal I$;
 \item there are sets $E_k\subseteq\Omega$ and maps
 $g_k:E_k\to\omega^k$ $(k\geq1)$ such that:
 (i) $E_1=\Omega$, $E_{k+1}\subseteq E_k$ and
 $\bigcap_kE_k=\varnothing$;
 (ii) $\Omega\setminus E_k\in\mathcal I$ for all $k$;
 (iii) $\pi_{k,k+1}\circ g_{k+1}=g_k$ on $E_{k+1}$;
 (iv) $g_k^{-1}[B]\in\mathcal I$ for every $B\in\Fin^{\otimes k}$.
\end{enumerate}
\end{lemma}

\begin{proof}
(a)$\Rightarrow$(b). Let $f:\Omega\to D$ be a reduction. Let
$\ell(x)$ be the level with $f(x)\in\omega^{\ell(x)}$ and put
$E_k=\{x:\ell(x)\geq k\}$, $g_k(x)=\pi_{k,\ell(x)}(f(x))$. The union
of the first $k-1$ levels of $D$ belongs to $\Fw$ (take $B=\varnothing$
in \eqref{eq:Fw}), so $\Omega\setminus E_k$, its $f$-preimage, is in
$\mathcal I$. Coherence (iii) holds because compositions of suffix
projections are suffix projections. For (iv), given
$B\in\Fin^{\otimes k}$, the set
$M_B=\bigcup_{m\geq k}\pi_{k,m}^{-1}[B]$ belongs to $\Fw$ by
\eqref{eq:Fw}, and $g_k^{-1}[B]=E_k\cap f^{-1}[M_B]\in\mathcal I$.
Each $\ell(x)$ is finite, so $\bigcap_kE_k=\varnothing$.

(b)$\Rightarrow$(a). For $x\in\Omega$ let
$\ell(x)=\max\{k:x\in E_k\}$, which exists by (i), and define
$f(x)=g_{\ell(x)}(x)\in\omega^{\ell(x)}\subseteq D$. Let
$M\in\Fw$ and choose $k$, $B$ as in \eqref{eq:Fw}. If $f(x)\in M$,
then either $x\notin E_k$, or $\ell(x)\geq k$ and, by coherence,
$g_k(x)=\pi_{k,\ell(x)}(f(x))\in
\pi_{k,\ell(x)}[M\cap\omega^{\ell(x)}]\subseteq B$. Hence
$f^{-1}[M]\subseteq(\Omega\setminus E_k)\cup g_k^{-1}[B]
\in\mathcal I$, and $f$ witnesses $\Fw\leq_K\mathcal I$.
\end{proof}

\begin{theorem}[coordinate criterion]\label{thm:criterion}
For an ideal $\mathcal I$ on a countable set $\Omega$ the following
are equivalent:
\begin{enumerate}[label=(\alph*)]
 \item $\Fw\leq_K\mathcal I$;
 \item there are functions $\xi_j:\Omega\to\omega$ $(j\geq1)$ such
 that for every $k$ the window
 $q_k=(\xi_k,\xi_{k-1},\dots,\xi_1):\Omega\to\omega^k$ witnesses
 $\Fin^{\otimes k}\leq_K\mathcal I$.
\end{enumerate}
A sequence as in (b) is called a \emph{tower} over $\mathcal I$.
\end{theorem}

\begin{proof}
(a)$\Rightarrow$(b). Take $(E_k,g_k)$ from
Lemma~\ref{lem:tower}. By coherence, $g_{k+1}$ extends $g_k$ by
prepending one outermost coordinate; define $\xi_j(x)$, for
$x\in E_j$, as the first coordinate of $g_j(x)$, and $\xi_j(x)=0$ for
$x\notin E_j$. Then $q_k\restr E_k=g_k$, so for
$B\in\Fin^{\otimes k}$,
\[
 q_k^{-1}[B]\subseteq g_k^{-1}[B]\cup(\Omega\setminus E_k)
 \in\mathcal I .
\]

(b)$\Rightarrow$(a). Fix an injection $e:\Omega\to\omega$ and put
$E_k=\{x:e(x)\geq k-1\}$, $g_k=q_k\restr E_k$. The complements
$\Omega\setminus E_k$ are finite, $\bigcap_kE_k=\varnothing$,
coherence is automatic (the $q_k$ are windows of one sequence), and
$g_k^{-1}[B]\subseteq q_k^{-1}[B]\in\mathcal I$. Now apply
Lemma~\ref{lem:tower}.
\end{proof}

\begin{lemma}[subtuples]\label{lem:subtuples}
If $q_k=(\xi_k,\dots,\xi_1)$ witnesses
$\Fin^{\otimes k}\leq_K\mathcal I$ and $j_1<\dots<j_r\leq k$, then
$(\xi_{j_r},\dots,\xi_{j_1})$ witnesses
$\Fin^{\otimes r}\leq_K\mathcal I$.
\end{lemma}

\begin{proof}
$(\xi_{j_r},\dots,\xi_{j_1})=\pi\circ q_k$, where $\pi$ is the
ordered projection of $\omega^k$ onto the positions
$k+1-j_r<\dots<k+1-j_1$ (counted from the outermost). For
$B\in\Fin^{\otimes r}$, Lemma~\ref{lem:insertion} gives
$\pi^{-1}[B]\in\Fin^{\otimes k}$, so
$(\pi\circ q_k)^{-1}[B]=q_k^{-1}[\pi^{-1}[B]]\in\mathcal I$.
\end{proof}

\begin{proposition}[all finite fragments are realizable]
\label{prop:fragments}
For every $N\geq1$ there are maps $q_k:\omega\to\omega^k$
$(1\leq k\leq N)$ commuting with the projections $\pi_{k,N}$ such that
each $q_k$ witnesses $\Fin^{\otimes k}\leq_K\Hw$; moreover all
preimages of small sets may be taken inside the single level
$\HH_{N-1}$.
\end{proposition}

\begin{proof}
By Theorem~\ref{thm:fubini} there is a bijection
$q_N:\omega\to\omega^N$ with
$X\in\HH_{N-1}\iff q_N[X]\in\Fin^{\otimes N}$; concretely,
\[
 q_N(x)=\bigl(b_x(0),\dots,b_x(N-2),\,e_{b_x\restr(N-1)}(x)\bigr),
\]
where $e_u:A_u\to\omega$ are fixed bijections of the level-$(N-1)$
cells. In particular $q_N^{-1}[B]\in\HH_{N-1}$ for
$B\in\Fin^{\otimes N}$. Put $q_k=\pi_{k,N}\circ q_N$. For
$B\in\Fin^{\otimes k}$, Lemma~\ref{lem:insertion} gives
$\pi_{k,N}^{-1}[B]\in\Fin^{\otimes N}$ (the last $k$ positions form
an ordered projection), so $q_k^{-1}[B]\in\HH_{N-1}$. Coherence with
the projections is automatic. We call this system the
\emph{canonical fragment of length $N$}.
\end{proof}

Consequently no refutation of $\Fw\leq_K\Hw$ can proceed through a
fixed finite number of levels: the failure, if any, must be a failure
of compactness. The rest of the section quantifies it exactly.

\subsection{Uniformization and essential depth}

For a pair $\varphi,\eta:\omega\to\omega$ and $h:\omega\to\omega$ we
write
\[
 W_h=\{x:\eta(x)\leq h(\varphi(x))\}
\]
for the \emph{graph tests} of the pair $(\varphi,\eta)$. If
$(\varphi,\eta)$ witnesses $\Fin^{\otimes2}\leq_K\Hw$
(outer coordinate $\varphi$), then every $W_h$ belongs to $\Hw$,
being the preimage of the graph
$\{(a,b):b\leq h(a)\}\in\Fin^{\otimes2}$.

\begin{lemma}[graph uniformization]\label{lem:uniformization}
If $(\varphi,\eta)$ witnesses $\Fin^{\otimes2}\leq_K\Hw$, then there
is $m<\omega$ with $W_h\in\HH_m$ for \emph{every} $h$. The analogue
holds in every dimension $k$: if $q_k=(\xi_k,\dots,\xi_1)$ witnesses
$\Fin^{\otimes k}\leq_K\Hw$, there is $m<\omega$ such that
$U_f=\{x:\xi_1(x)\leq f(\xi_k(x),\dots,\xi_2(x))\}\in\HH_m$ for every
$f:\omega^{k-1}\to\omega$. The proof uses only that $\Hw$ is an
increasing union of ideals.
\end{lemma}

\begin{proof}
Dimension $2$. The fibres $R_i=\varphi^{-1}\{i\}$ lie in $\Hw$
(preimages of $\{i\}\times\omega\in\Fin^{\otimes2}$), and so does
every $W_h$. Suppose there is no uniform bound. For each $n\geq1$ the
set $P_n=\bigcup_{i<n}R_i$ lies in $\Hw$; pick $r_n$ with
$P_n\in\HH_{r_n}$, and, using the failure of uniformity, pick $h_n$
with $W_{h_n}\notin\HH_{\max(n,r_n)}$. Then
$V_n=W_{h_n}\setminus P_n\notin\HH_n$: otherwise both $V_n$ and
$W_{h_n}\cap P_n$ would lie in $\HH_{\max(n,r_n)}$, hence so would
their union $W_{h_n}$. Define
\[
 h(i)=\max\{h_n(i):n\leq i\},
\]
a finite maximum. For each $n$ and $x\in V_n$ we have
$\varphi(x)\geq n$ (as $x\notin P_n$), so
$h(\varphi(x))\geq h_n(\varphi(x))\geq\eta(x)$ and $x\in W_h$. Thus
$W_h\supseteq V_n$ for every $n$, so $W_h\notin\HH_n$ for all $n$,
contradicting $W_h\in\Hw=\bigcup_n\HH_n$.

Dimension $k$. Fix an enumeration $\langle p_i:i\in\omega\rangle$ of
$\omega^{k-1}$. The fibres
$R_{p}=\{x:(\xi_k(x),\dots,\xi_2(x))=p\}$ lie in $\Hw$ (preimages of
$\{p\}\times\omega$, which is in $\Fin^{\otimes k}$ by
Lemma~\ref{lem:insertion}), and each $U_f$ lies in $\Hw$ (its image
set has all innermost sections finite, hence lies in
$\Fin^{\otimes k}$ by induction on $k$ as in
Lemma~\ref{lem:insertion}). Now repeat the diagonal above with $R_i$
replaced by $R_{p_i}$: if there is no uniform bound, choose $r_n$,
$f_n$ as before and set $f(p)=\max\{f_n(p):n\leq i(p)\}$, where
$i(p)$ is the index of $p$; the same computation shows
$U_f\supseteq U_{f_n}\setminus\bigcup_{i<n}R_{p_i}\notin\HH_n$ for
all $n$, a contradiction.
\end{proof}

\begin{lemma}[fronts]\label{lem:fronts}
Let $U\subseteq\omega^{<\omega}$ be downward closed and $m\geq1$.
Then $\varnothing\notin D^m_{\Fin}(U)$ if and only if
$U\cap\omega^m\in\Fin^{\otimes m}$.
\end{lemma}

\begin{proof}
Induction on $m$. For $m=1$: $\varnothing\notin D_{\Fin}(U)$ iff the
root has finitely many children in $U$ iff $U\cap\omega^1$ is finite.
For the step, write $U_n=\{t:\langle n\rangle\concat t\in U\}$
(downward closed). By Lemma~\ref{lem:persistence}(b),
$\langle n\rangle\in D^m_{\Fin}(U)$ iff
$\varnothing\in D^m_{\Fin}(U_n)$, and by the successor equivalence in
the proof of Theorem~\ref{thm:fubini} (with $\mathcal J=\Fin$),
$\varnothing\notin D^{m+1}_{\Fin}(U)$ iff
$\{n:\langle n\rangle\in D^m_{\Fin}(U)\}$ is finite. By the induction
hypothesis this set equals
$\{n:U_n\cap\omega^m\notin\Fin^{\otimes m}\}$ up to a finite set, and
its finiteness is precisely
$U\cap\omega^{m+1}\in\Fin\otimes\Fin^{\otimes m}
=\Fin^{\otimes(m+1)}$.
\end{proof}

\begin{lemma}[local--global propagation]\label{lem:localglobal}
For $X\subseteq\omega$ and $d\geq1$ let
\[
 P_d(X)=\{s\in\omega^d:X\cap A_s\notin\Hw(s)\},
\]
where $\Hw(s)$ is the localized nested limit of the subtree at $s$.
\begin{enumerate}[label=(\alph*)]
 \item If $X\in\Hw$ then $P_d(X)\in\Fin^{\otimes d}$. Consequently,
 if $P\subseteq\omega^d$ is $\Fin^{\otimes d}$-positive and
 $X_s\subseteq A_s$ satisfies $X_s\notin\Hw(s)$ for each $s\in P$,
 then $\bigcup_{s\in P}X_s\notin\Hw$.
 \item If $P\in\Fin^{\otimes d}$ then
 $\bigcup_{s\in P}A_s\in\HH_d$.
\end{enumerate}
\end{lemma}

\begin{proof}
(a) Let $X\in\HH_m$. If $m=0$, then $X$ is finite and
$P_d(X)=\varnothing$ for every $d\geq1$. Assume $m\geq1$ and induct
on $d$. For $d=1$, the Fubini formula
(Theorem~\ref{thm:fubini}) gives
$X\cap A_{\langle n\rangle}\in\HH_{m-1}(\langle n\rangle)
\subseteq\Hw(\langle n\rangle)$ for all but finitely many $n$, so
$P_1(X)$ is finite. For $d+1$: the sections of $P_{d+1}(X)$ satisfy
$(P_{d+1}(X))_n=P_d(X\cap A_{\langle n\rangle})$, computed inside the
localized partition of $A_{\langle n\rangle}$. For all but finitely
many $n$ we have
$X\cap A_{\langle n\rangle}\in\HH_{m-1}(\langle n\rangle)$, and for
those $n$ the induction hypothesis (via local homogeneity,
Corollary~\ref{cor:homogeneity}) puts the section in
$\Fin^{\otimes d}$; the finitely many exceptional sections are
tolerated by $\Fin\otimes\Fin^{\otimes d}$.

The consequence is the contrapositive: if
$Y=\bigcup_{s\in P}X_s\in\Hw$, then, since
$X_s\subseteq Y\cap A_s$ and ideals are downward closed,
$P\subseteq P_d(Y)\in\Fin^{\otimes d}$.

(b) Induction on $d$. For $d=1$, $P$ is finite and the root of the
trace tree of $Z=\bigcup_{s\in P}A_s$ has finitely many children, so
$Z\in\HH_1$. For $d+1$: all but finitely many sections $P_n$ of $P$
lie in $\Fin^{\otimes d}$, and for those the induction hypothesis
gives
$Z\cap A_{\langle n\rangle}
=\bigcup_{t\in P_n}A_{\langle n\rangle\concat t}
\in\HH_d(\langle n\rangle)$; by the Fubini formula
$Z\in\HH_{d+1}$.
\end{proof}

\begin{remark}\label{rem:Hs}
In this section all local positivity refers to the localized nested
limit $\Hw(s)$, not to the full transfinite ideal $\HH(s)$; the
latter is a larger ideal, so positivity with respect to it is a
stronger condition, and the Laver-tree criterion of
Theorem~\ref{thm:nucleus} characterizes positivity for $\HH(s)$ only.
The two usages are compatible in the proofs below: sets that are
$\HH(s)$-positive (for instance, the full cells $A_s$) are in
particular $\Hw(s)$-positive.
\end{remark}

\begin{lemma}[capture]\label{lem:capture}
For all $\varphi,\eta:\omega\to\omega$ there is $h$ such that
$U_\varphi:=\{s:\varphi[A_s]\ \text{is infinite}\}\subseteq
T^\infty_{W_h}$.
\end{lemma}

\begin{proof}
Enumerate the nodes of $U_\varphi$ with infinite repetition. At each
visit to $s$, choose a point $x\in A_s$ whose value $\varphi(x)$ has
not appeared in any earlier choice; this is possible because
$\varphi[A_s]$ is infinite. Let $Y$ be the set of chosen points; then
$\varphi$ is injective on $Y$ and $Y\cap A_s$ is infinite for every
$s\in U_\varphi$. Define $h(\varphi(y))=\eta(y)$ for $y\in Y$
(well defined by injectivity) and $h=0$ on all other values. Then
$Y\subseteq W_h$, so
$U_\varphi\subseteq T^\infty_Y\subseteq T^\infty_{W_h}$.
\end{proof}

\begin{definition}[essential depth]\label{def:depth}
For $\varphi:\omega\to\omega$,
\[
 \delta(\varphi)=\min\bigl\{d\geq1:\exists Z\in\Hw\ \forall
 s\in\omega^d\ \ \varphi[A_s\setminus Z]\ \text{is finite}\bigr\},
\]
with $\delta(\varphi)=\infty$ if no such $d$ exists. Essential depth
is invariant under modification of $\varphi$ on a set in $\Hw$: a
modification on $Z_0\in\Hw$ replaces a witness $Z$ by
$Z\cup Z_0$.
\end{definition}

\begin{remark}
The restriction $d\geq1$ excludes only a degenerate case: if
$\varphi$ is an outer coordinate of a witness of
$\Fin^{\otimes2}\leq_K\Hw$ and $\varphi$ had finite range off some
$Z\in\Hw$, then $\omega$ would be the union of $Z$ with finitely many
fibres of $\varphi$, all of which lie in $\Hw$, contradicting
properness.
\end{remark}

\begin{corollary}[finiteness of the depth]\label{cor:depthfinite}
If $(\varphi,\eta)$ witnesses $\Fin^{\otimes2}\leq_K\Hw$ and all its
graph tests lie in $\HH_m$ with $m\geq1$, then
$\delta(\varphi)\leq m$. In particular $\delta(\varphi)<\omega$.
\end{corollary}

\begin{proof}
Take the $h$ given by Lemma~\ref{lem:capture} for the pair
$(\varphi,\eta)$. Since $W_h\in\HH_m$, the root does not survive $m$
derivatives of $T^\infty_{W_h}$; by
Lemma~\ref{lem:persistence}(a) and
$U_\varphi\subseteq T^\infty_{W_h}$, the root does not survive $m$
derivatives of the downward closed tree $U_\varphi$ either (note
$\varphi[A_{s\restr j}]\supseteq\varphi[A_s]$, so $U_\varphi$ is
indeed a tree). By Lemma~\ref{lem:fronts},
$P=U_\varphi\cap\omega^m\in\Fin^{\otimes m}$, and by
Lemma~\ref{lem:localglobal}(b),
$Z=\bigcup_{s\in P}A_s\in\HH_m\subseteq\Hw$. For $s\in\omega^m$: if
$s\in P$ then $A_s\setminus Z=\varnothing$; if $s\notin P$ then
$\varphi[A_s]$ is finite. Hence $Z$ witnesses
$\delta(\varphi)\leq m$. For the last clause: if each $W_h$ is merely
in $\Hw$, apply the argument to the single $h$ of
Lemma~\ref{lem:capture}, choosing $m\geq 1$ with $W_h\in\HH_m$.
\end{proof}

\subsection{Monotonicity and dimension}

\begin{theorem}[monotonicity of essential depth]\label{thm:monotone}
If $(\varphi,\eta)$ witnesses $\Fin^{\otimes2}\leq_K\Hw$ and
$\delta(\eta)<\omega$, then
\[
 \delta(\varphi)\leq\delta(\eta).
\]
\end{theorem}

\begin{proof}
Let $d=\delta(\eta)$ and let $Z_\eta\in\Hw$ satisfy:
$\eta[A_s\setminus Z_\eta]$ is finite for every $s\in\omega^d$.
Redefining $\eta$ as $0$ on $Z_\eta$ preserves the reduction (any new
preimage differs from the old one within $Z_\eta$) and does not
change either essential depth, so we may assume that $\eta[A_s]$ is
finite for every $s\in\omega^d$. Write
\[
 M_s=\max\eta[A_s]\qquad(s\in\omega^d),
\]
fix an injection $e:\omega^d\to\omega$, and define
\[
 h(a)=\max\bigl(\{M_s:e(s)<a\}\cup\{0\}\bigr),
\]
a finite maximum since $e$ is injective. We claim that the single
graph test $W_h$ witnesses $\delta(\varphi)\leq d$: for every
$s\in\omega^d$,
\[
 \varphi[A_s\setminus W_h]\subseteq[0,e(s)].
\]
Indeed, if $x\in A_s$ and $\varphi(x)>e(s)$, then $M_s$ enters the
maximum defining $h(\varphi(x))$, so
$\eta(x)\leq M_s\leq h(\varphi(x))$ and $x\in W_h$. Since
$W_h\in\Hw$, the set $Z=W_h$ is the required witness.
\end{proof}

The point of the diagonal $e$ is that all the local residues are
absorbed by a \emph{single} graph test: no cell-by-cell analysis of
the fibres of $\varphi$ is needed.

\begin{theorem}[dimension]\label{thm:dimension}
Let $\Psi=(\psi_r,\dots,\psi_1):\omega\to\omega^r$ witness
$\Fin^{\otimes r}\leq_K\Hw$. If $\delta(\psi_i)\leq d$ for all
$i\leq r$, then $r\leq d$.
\end{theorem}

\begin{proof}
Normalize the $r$ coordinates simultaneously: the union of their $r$
witnesses lies in $\Hw$, and redefining each $\psi_i$ as $0$ there
preserves the reduction, so we may assume $\Psi[A_s]$ is finite for
every $s\in\omega^d$.

For each $s\in\omega^d$, the cell $A_s$ is the full local space, so
$A_s\notin\Hw(s)$; since $A_s$ is a finite union of fibres of $\Psi$,
some value $p(s)\in\Psi[A_s]$ has
\[
 P_s=A_s\cap\Psi^{-1}\{p(s)\}\notin\Hw(s).
\]
The function $p:\omega^d\to\omega^r$ witnesses
$\Fin^{\otimes r}\leq_K\Fin^{\otimes d}$: if $B\in\Fin^{\otimes r}$
and $p^{-1}[B]\notin\Fin^{\otimes d}$, then
Lemma~\ref{lem:localglobal}(a) makes
$\bigcup_{s\in p^{-1}[B]}P_s$ $\Hw$-positive, yet this union is
contained in $\Psi^{-1}[B]\in\Hw$, a contradiction.

Finally, $\Fin^{\otimes r}\leq_K\Fin^{\otimes d}$ implies $r\leq d$:
if $r>d$, then composing the upward witnesses of
Corollary~\ref{cor:chain} (through the isomorphisms of
Theorem~\ref{thm:fubini}) gives
$\Fin^{\otimes(d+1)}\leq_K\Fin^{\otimes r}\leq_K\Fin^{\otimes d}$,
contradicting the strictness in Corollary~\ref{cor:chain}.
\end{proof}

\subsection{Non-extension and the main theorem}

\begin{theorem}[sharp non-extension]\label{thm:nonext}
Let $(\xi_N,\dots,\xi_1)$ be such that every window $q_k$ $(k\leq N)$
witnesses $\Fin^{\otimes k}\leq_K\Hw$, and suppose all graph tests of
the pair $(\xi_2,\xi_1)$ lie in $\HH_m$, $m\geq1$. Then
\[
 N\leq m+1 .
\]
In particular the pair $(\xi_2,\xi_1)$ does not extend to a coherent
window of length $m+2$.
\end{theorem}

\begin{proof}
By Corollary~\ref{cor:depthfinite} applied to the window $q_2$,
\[
 \delta(\xi_2)\leq m .
\]
For each $j$ with $3\leq j\leq N$, the pair $(\xi_j,\xi_2)$ witnesses
$\Fin^{\otimes2}\leq_K\Hw$ (Lemma~\ref{lem:subtuples} with
$j_1=2<j_2=j$), so Theorem~\ref{thm:monotone} gives
\[
 \delta(\xi_j)\leq\delta(\xi_2)\leq m .
\]
Thus $\delta(\xi_j)\leq m$ for all $2\leq j\leq N$. The subtuple
$Q=(\xi_N,\dots,\xi_2)$, of length $N-1$, witnesses
$\Fin^{\otimes(N-1)}\leq_K\Hw$ (Lemma~\ref{lem:subtuples} again), and
all its coordinates have essential depth at most $m$;
Theorem~\ref{thm:dimension} yields $N-1\leq m$.
\end{proof}

\begin{proposition}[calibration: the bound is attained]
\label{prop:calibration}
For $K\geq2$, the canonical fragment
$(\xi_K,\dots,\xi_1)$ of length $K$
(Proposition~\ref{prop:fragments}):
\begin{enumerate}[label=(\alph*)]
 \item every graph test of $(\xi_2,\xi_1)$ lies in $\HH_{K-1}$;
 \item $\delta(\xi_2)=K-1$; consequently $K-1$ is the \emph{least}
 uniform bound $m$ for the graph tests of $(\xi_2,\xi_1)$.
\end{enumerate}
Hence a pair with parameter $m=K-1$ does extend to a coherent window
of length $m+1=K$, and by Theorem~\ref{thm:nonext} not to length
$m+2$: the bound $N(m)=m+2$ is sharp.
\end{proposition}

\begin{proof}
Recall the canonical fragment: $\xi_j(x)=b_x(K-j)$ for
$2\leq j\leq K$, and $\xi_1(x)=e_{b_x\restr(K-1)}(x)$.

(a) $W_h=q_K^{-1}[G_h]$ where
$G_h=\{t\in\omega^K:t(K-1)\leq h(t(K-2))\}=\pi^{-1}
[\{(a,b):b\leq h(a)\}]$ for the ordered projection $\pi$ onto the
last two positions. The graph $\{(a,b):b\leq h(a)\}$ has all sections
finite, so it lies in $\Fin^{\otimes2}$, and
Lemma~\ref{lem:insertion} gives $G_h\in\Fin^{\otimes K}$; hence
$W_h\in\HH_{K-1}$, uniformly in $h$.

(b) Upper bound: $\xi_2(x)=b_x(K-2)$ is constant on every cell of
level $K-1$, so $Z=\varnothing$ witnesses
$\delta(\xi_2)\leq K-1$.

Lower bound: suppose $d\leq K-2$ and $Z\in\Hw$ satisfy
$\xi_2[A_s\setminus Z]\subseteq F_s$ with $F_s$ finite, for every
$s\in\omega^d$. Fix any $u\in\omega^{K-2}$ and any
$v\notin F_{u\restr d}$; every $x\in A_{u\concat v}$ has
$\xi_2(x)=b_x(K-2)=v$, so the \emph{entire cell}
$A_{u\concat v}$ is contained in $Z$. Consequently $T^\infty_Z$
contains the tree
\[
 R=\omega^{\leq K-2}\ \cup\
 \{u\concat v\concat w:
 u\in\omega^{K-2},\ v\notin F_{u\restr d},\ w\in\omega^{<\omega}\}:
\]
below every node of level $K-2$ all but finitely many children
survive, with full cones above them, and every node of level $<K-2$
has all its children. Thus $R$ is a ($\Fin$-positive) Laver tree
inside $T^\infty_Z$, so $Z\notin\HH\supseteq\Hw$ by
Theorem~\ref{thm:nucleus} --- contradicting $Z\in\Hw$. Hence
$\delta(\xi_2)\geq K-1$.

Minimality of the bound: any uniform bound $m$ for the graph tests
satisfies $\delta(\xi_2)\leq m$ by
Corollary~\ref{cor:depthfinite}, so $m\geq K-1$, and (a) shows
$m=K-1$ is achieved.
\end{proof}

\begin{theorem}[Main Theorem]\label{thm:main}
$\Fw\not\leq_K\Hw$.
\end{theorem}

\begin{proof}
Suppose $\Fw\leq_K\Hw$. By Theorem~\ref{thm:criterion} there is a
tower $(\xi_j)_{j\geq1}$ over $\Hw$. By
Lemma~\ref{lem:uniformization} there is $m\geq1$ such that all graph
tests of the pair $(\xi_2,\xi_1)$ lie in $\HH_m$. But the window
$q_{m+2}=(\xi_{m+2},\dots,\xi_1)$ is a coherent extension of
$(\xi_2,\xi_1)$ of length $m+2$, all of whose windows witness the
corresponding reductions --- contradicting
Theorem~\ref{thm:nonext}.
\end{proof}

Together with Proposition~\ref{prop:fragments} and
Proposition~\ref{prop:calibration}, the failure is exactly
calibrated: every finite fragment of a tower exists, the canonical
fragment of length $K$ realizes the extremal parameter $m=K-1$, and
no single sequence of coordinates realizes all lengths
simultaneously. This is a failure of compactness, not of any finite
window.

\subsection{Consequences}

\begin{definition}\label{def:trace}
An ideal $\mathcal I$ on $\omega$ is \emph{trace-determined} if
$T^\infty_X=T^\infty_Y$ implies
$(X\in\mathcal I\iff Y\in\mathcal I)$.
\end{definition}

\begin{lemma}[clone partition]\label{lem:clones}
There are a partition $\langle P_x:x\in\omega\rangle$ of $\omega$
into infinite sets and a finite-to-one $d:\omega\to\omega$ such that
$P_x\subseteq A_{b_x\restr d(x)}$ for every $x$.
\end{lemma}

\begin{proof}
Fix any finite-to-one $d$ (say, a bijection) and put
$C_x=A_{b_x\restr d(x)}$, an infinite set containing $x$. Enumerate
$\omega\times\omega$ as $\langle(x_j,k_j):j<\omega\rangle$ and choose
recursively $q_j\in C_{x_j}\setminus\{q_i:i<j\}$ (possible: only
finitely many points are used at each stage). The sets
$Q_x=\{q_j:x_j=x\}$ are infinite and pairwise disjoint, with
$Q_x\subseteq C_x$. Distribute the remainder
$R=\omega\setminus\bigcup_xQ_x$ by setting
$P_x=Q_x\cup(\{x\}\cap R)$; since $x\in C_x$, still
$P_x\subseteq C_x$, and $\langle P_x\rangle$ is a partition of
$\omega$ into infinite sets.
\end{proof}

\begin{lemma}[shadow]\label{lem:shadow}
For $A\subseteq\omega$ let
$\sh(A)=\{x:A\cap P_x\neq\varnothing\}$. Then there is
$Z\subseteq A$ with exactly one point in each $P_x$, $x\in\sh(A)$,
such that $T^\infty_Z=T^\infty_{\sh(A)}$. Consequently, if
$\mathcal I$ is trace-determined, then
$A\in\mathcal I\Rightarrow\sh(A)\in\mathcal I$.
\end{lemma}

\begin{proof}
For $x\in\sh(A)$ pick $z_x\in A\cap P_x$ and put
$Z=\{z_x:x\in\sh(A)\}$; the $z_x$ are distinct because the $P_x$ are
disjoint. Fix a node $s$ of length $m$. If $d(x)\geq m$, then
$z_x\in P_x\subseteq A_{b_x\restr d(x)}\subseteq A_{b_x\restr m}$, so
$x$ and $z_x$ lie in the same level-$m$ cell. Since
$E_m=\{x:d(x)<m\}$ is finite, the bijection $x\mapsto z_x$ shows
$|\sh(A)\cap A_s|=\omega$ iff $|Z\cap A_s|=\omega$ up to finitely
many exceptions, so
$T^\infty_Z=T^\infty_{\sh(A)}$. For the consequence: $Z\subseteq A$
gives $Z\in\mathcal I$, and equality of traces gives
$\sh(A)\in\mathcal I$.
\end{proof}

\begin{theorem}[$\mathrm{Kat}$ property of trace ideals]
\label{thm:katH}
Every tall trace-determined ideal on $\omega$ has property
$\mathrm{Kat}$. In particular $\Hw$, each $\HH^{\mathcal J}_\alpha$
with $\alpha\geq1$ (for every proper $\mathcal J\supseteq\Fin$), the
full ideal $\HH^{\mathcal J}$, and $\HH$ have property
$\mathrm{Kat}$.
\end{theorem}

\begin{proof}
Let $\mathcal I$ be tall and trace-determined, and let
$f:\operatorname{dom}(\mathcal K)\to\omega$ witness
$\mathcal I\leq_K\mathcal K$. Fix bijections $\pi_x:\omega\to P_x$
and, for each $x$, an injective enumeration $e_x$ of the fibre
$f^{-1}\{x\}$. Define
\[
 g(t)=\pi_{f(t)}\bigl(e_{f(t)}(t)\bigr).
\]
Then $g$ is injective, and for $A\in\mathcal I$: if $g(t)\in A$ then
$A\cap P_{f(t)}\neq\varnothing$, i.e.\ $f(t)\in\sh(A)$; by
Lemma~\ref{lem:shadow}, $\sh(A)\in\mathcal I$, so
\[
 g^{-1}[A]\subseteq f^{-1}[\sh(A)]\in\mathcal K .
\]
Thus $g$ is an injective reduction. To make it bijective: since
$\mathcal I$ is tall and $g[\operatorname{dom}(\mathcal K)]$ is
infinite, there is an infinite
$C\in\mathcal I$ with
$C\subseteq g[\operatorname{dom}(\mathcal K)]$. The set
$E=g^{-1}[C]$ lies in $\mathcal K$. Choose a bijection
$q:E\to C\cup(\omega\setminus g[\operatorname{dom}(\mathcal K)])$ and
let $h=g$ off $E$ and $h=q$ on $E$; then $h$ is a bijection onto
$\omega$ and, for $A\in\mathcal I$,
$h^{-1}[A]\subseteq g^{-1}[A]\cup E\in\mathcal K$. Hence
$\mathcal I\sqsubseteq\mathcal K$.

The particular cases: membership in each of the ideals listed is
defined from $T^\infty_X$ alone, and each is tall
(Theorem~\ref{thm:hierarchy} for the levels $\alpha\geq1$; the full
ideals and $\Hw$ contain $\HH^{\mathcal J}_1$).
\end{proof}

\begin{corollary}\label{cor:nocopy}
$\Hw$ contains no isomorphic copy of $\Fw$
($\Fw\not\sqsubseteq\Hw$). By \ref{K:katFw}, the two non-embedding
statements $\Fw\not\leq_K\Hw$ and $\Fw\not\sqsubseteq\Hw$ are in fact
equivalent.
\end{corollary}

\begin{proof}
A bijective witness is in particular a Kat\v etov witness, so
Theorem~\ref{thm:main} forbids copies; conversely, by
\ref{K:katFw} any Kat\v etov reduction of $\Fw$ upgrades to a copy.
\end{proof}

\begin{corollary}\label{cor:kwela46}
$\Fw\not\leq_K\Fwp$, and hence $\Fw\not\sqsubseteq\Fwp$.
\end{corollary}

\begin{proof}
$\Fwp\sqsubseteq\Hw$ (Proposition~\ref{prop:nestedlimit}) gives in
particular $\Fwp\leq_K\Hw$; if $\Fw\leq_K\Fwp$, transitivity would
give $\Fw\leq_K\Hw$, contradicting Theorem~\ref{thm:main}. This
conclusion is already a consequence of Kwela's Theorem~4.6
\cite{Kwela} together with the property $\mathrm{Kat}$ of $\Fw$ from
\cite[Proposition~5.1]{Barbarski}; the present argument is an independent route through the nested
limit: it neither computes the exact Borel rank of either limit ideal
nor uses independence of the partitions, and it yields the
Kat\v etov form directly. Its only rank input is the standard
strictness of the finite Fubini powers, used in
Corollary~\ref{cor:chain}.
\end{proof}

\begin{remark}[three amalgamations]\label{rem:threelimits}
Consider the three natural ways of amalgamating all finite Fubini powers
over a single countable set: the canonical (suffix-coherent) limit
$\Fw$, the independent-partitions limit $\Fwp$, and the nested tree
limit $\Hw$. The canonical limit is distinguished from each of the
other two: we have
$\Fwp\sqsubseteq\Hw$ (Proposition~\ref{prop:nestedlimit}), while
$\Fw$ reduces to neither (Theorem~\ref{thm:main},
Corollary~\ref{cor:kwela46}), even though all its finite fragments
are realizable over both. The present results do not distinguish
$\Fwp$ from $\Hw$: the comparison $\Hw\leq_K\Fwp$ remains open, as
does $\Hw\leq_K\Fw$ (Problem~\ref{prob:limits}).
\end{remark}

\section{Chromatic ideals of splitting levels}\label{sec:Gk}

This section applies the machinery to a second family of ideals
canonically attached to the tree partition. The family cannot even be
defined without the partition --- the coloring below is built from
$\Delta$, the separation-level function of the branch dictionary ---
and the derived hierarchy of Section~\ref{sec:DJ} measures it
exactly: the whole family lies inside the second level but not inside
the first,
$\GG_k\subsetneq\HH_2$ and $\GG_k\not\subseteq\HH_1$
(Theorems~\ref{thm:GH2} and~\ref{thm:H1PC}), and the proof of the
upper location runs the derivative itself.

\begin{definition}\label{def:Gk}
For $k\geq2$, $c_k(\{x,y\})=\Delta(x,y)\bmod k$, and
$\GG_k=\langle\HomC(c_k)\rangle$, the ideal generated by the
$c_k$-homogeneous sets.
\end{definition}

By \ref{K:colors}, once $\GG_k$ is proper it is $F_\sigma$ and tall
(properness is part of Theorem~\ref{thm:H1PC}).

\begin{theorem}\label{thm:GH2}
$\GG_k\subseteq\HH_2$ and $\GG_k\not\subseteq\HH_1$ for all $k\geq2$.
\end{theorem}

\begin{proof}
Let $Y$ be homogeneous of color $r$. If $|s|\not\equiv r\pmod k$,
two points of $Y$ lying in distinct children of $s$ would have
separation level $|s|$, which is impossible; hence $Y$ meets at most
one child of each such $s$. If $r\neq0$, this applies to the root, so
$\suc_{T^\infty_Y}(\varnothing)$ has at most one element and the root
dies at the first derivative: $Y\in\HH_1$. If $r=0$, it applies to
every node of length $1$ (as $k\geq2$), so all such nodes die at the
first derivative and the root dies at the second: $Y\in\HH_2$. Since
$\HH_2$ is an ideal containing all generators, $\GG_k\subseteq\HH_2$.

For the second claim: for each $n$ choose a node
$p_n\supseteq\langle n\rangle$ of length $k$ and points
$x_{n,j}\in A_{p_n\concat j}$ ($j\in\omega$), all distinct. The set
$X=\{x_{n,j}:n,j\in\omega\}$ is homogeneous of color $0$: two points
in the same row separate at level $k$, two points in different rows
at level $0$. Yet all first-level sections of $X$ are infinite, so
the root of $T^\infty_X$ has all its children and survives the first
derivative: $X\notin\HH_1$.
\end{proof}

\begin{lemma}[fusion with prescribed levels]\label{lem:fusion}
Let $q\geq2$ and let $L\subseteq\omega$ be infinite such that at least
two residues modulo $q$ occur infinitely often in $L$; let $R$ be the
set of residues occurring infinitely often. There are a pruned tree
$T\subseteq\omega^{<\omega}$ and a countable
$D\subseteq B\cap[T]$ such that: $[T]$ is a Cantor set and $D$ is
dense in $[T]$; the root of $T$ has a single successor and every node
at most two; every splitting of $T$ occurs at a level of $L$; and in
every relative open subset of $[T]$ splittings occur at every residue
class of $R$.
\end{lemma}

\begin{proof}
We build a countable set $D=\{d_0,d_1,\dots\}\subseteq B$ of
\emph{committed} branches by recursion, together with an increasing
sequence of \emph{splitting levels}. Fix an agenda: an enumeration of
$\omega\times R$ in which every pair $(i,r)$ appears infinitely often
and is treated only after $d_i$ has been committed. Start by
committing an arbitrary $d_0\in B$. At the stage treating $(i,r)$:
choose $\ell\in L$ with $\ell\equiv r\pmod q$, $\ell\geq1$ and
$\ell$ strictly larger than every level used so far (possible, as $r$
recurs in $L$); choose a value $v\neq d_i(\ell)$ and a new branch
$e\in B\cap[(d_i\restr\ell)\concat v]$ (possible, as $B$ meets every
cylinder in an infinite set); commit $e$.

Let $T$ be the tree of prefixes of $D$. By induction, all pairwise
separations of committed branches lie in $L$: the new branch $e$
separates from $d_i$ exactly at $\ell\in L$, and from any other
committed $d'$ at $\Delta(e,d')=\Delta(d_i,d')$ (an earlier level, by
the ultrametric inequality, since $\Delta(d_i,d')<\ell$).

Each splitting event uses a fresh level, so it creates exactly one
new splitting node (the node $d_i\restr\ell$, which acquires a second
direction) and no node is ever split twice: every node of $T$ has at
most two successors, and the root exactly one (all levels used are
$\geq1$). Thus $T$ is pruned (every node lies on a committed branch)
and finitely branching, so $[T]$ is compact; moreover
$[T]=\overline D$: every prefix of a point of $[T]$ is a prefix of a
committed branch. Since the agenda revisits every committed branch
infinitely often, every basic open subset of $[T]$ contains two
distinct branches, so $[T]$ is perfect --- a Cantor set --- and $D$
is dense.

Finally, let $[t]\cap[T]\neq\varnothing$ be a relative open set and
$r\in R$. The node $t$ lies on some committed $d_i$. The pair $(i,r)$
recurs in the agenda and the fresh levels increase without bound, so
one of its later treatments occurs at a level $\ell>|t|$. This
produces a splitting inside $[t]\cap[T]$ at a level
$\equiv r\pmod q$.
\end{proof}

\begin{theorem}\label{thm:H1PC}
For every $k\geq2$ there is $X\in\HH_1\cap\PC\setminus\GG_k$. In
particular $\GG_k$ is proper and $\GG_k\subsetneq\HH_2$.
\end{theorem}

\begin{proof}
Apply Lemma~\ref{lem:fusion} with $L=\omega$ and $q=k$ (so $R$ is the
full set of residues), and let $X=\{x:b_x\in D\}$. Then $T_X=T$: the
root of $T^\infty_X$ has at most one child, so $X\in\HH_1$; and
$|\suc_{T_X}(s)|\leq2\leq|s|+1$ for $|s|\geq1$, with a single
successor at the root, so $X\in\PC$.

Suppose $X\subseteq H_0\cup\dots\cup H_{m-1}\cup F$ with each $H_i$
homogeneous of color $r_i$ and $F$ finite. Consider the closures
$K_i=\overline{b[H_i\cap X]}$ in $[T]$. Any two distinct points
$p,q\in K_i$ satisfy $\Delta(p,q)\equiv r_i\pmod k$: taking
$x,y\in H_i\cap X$ with $b_x\restr(d+1)=p\restr(d+1)$ and
$b_y\restr(d+1)=q\restr(d+1)$, where $d=\Delta(p,q)$, gives
$\Delta(x,y)=d$, hence $d\equiv r_i$. If some $K_i$ contained a
nonempty relative open set, that open set would contain splittings in
at least two residue classes modulo $k$ (Lemma~\ref{lem:fusion} with
$R$ full and $k\geq2$), that is, pairs of branches with separation
levels in two distinct classes --- contradicting the previous
sentence. Hence every $K_i$ is closed and nowhere dense in $[T]$, and
$b[F]$ is finite, hence nowhere dense in the perfect set $[T]$. But
$[T]=\overline{b[X]}\subseteq\bigcup_iK_i\cup\overline{b[F]}$, a
finite union of nowhere dense closed sets covering a compact metric
space --- contradicting the Baire category theorem. So
$X\notin\GG_k$; in particular $\omega\notin\GG_k$ ($\GG_k$ is
proper), and $X\in\HH_2\setminus\GG_k$ gives
$\GG_k\subsetneq\HH_2$.
\end{proof}

\begin{theorem}[inclusion $=$ divisibility]\label{thm:divisibility}
For $k,\ell\geq2$:
$\GG_k\subseteq\GG_\ell\iff\ell\mid k$.
\end{theorem}

\begin{proof}
If $\ell\mid k$, every $c_k$-homogeneous set of color $r$ is
$c_\ell$-homogeneous of color $r\bmod\ell$, so
$\GG_k\subseteq\GG_\ell$. If $\ell\nmid k$, let
$L=\{jk:j\geq1\}$. The residues of $L$ modulo $\ell$ form the
subgroup generated by $k$ in $\mathbb Z_\ell$, of size
$\ell/\gcd(k,\ell)>1$, each occurring infinitely often. Apply
Lemma~\ref{lem:fusion} with this $L$ and $q=\ell$, and let
$X=\{x:b_x\in D\}$. All separations within $X$ lie in $L$, hence are
$\equiv0\pmod k$: $X$ is $c_k$-homogeneous of color $0$ and
$X\in\GG_k$. The Baire argument of Theorem~\ref{thm:H1PC}, now with
the residues modulo $\ell$ (at least two recur), gives
$X\notin\GG_\ell$.
\end{proof}

\begin{proposition}\label{prop:PCincomp}
$\GG_k$ and $\PC$ are incomparable under inclusion.
\end{proposition}

\begin{proof}
$\PC\not\subseteq\GG_k$ is Theorem~\ref{thm:H1PC}. For
$\GG_k\not\subseteq\PC$ we repeat the fusion in an $\omega$-ary
version. Build committed branches $d_i\in B$ and \emph{active nodes}
by recursion, with strictly increasing fresh levels, along an agenda
that repeats each of the following tasks infinitely often: (i) for
each committed $d_i$: pick a fresh level $\ell\equiv0\pmod k$, declare
the node $d_i\restr\ell$ active, and commit a new branch through
$(d_i\restr\ell)\concat v$ for a new value $v\neq d_i(\ell)$; (ii)
for each active node $u$: commit a new branch through
$u\concat v$ for a direction $v$ not yet used at $u$. As in
Lemma~\ref{lem:fusion}, all pairwise separations happen at the
splitting nodes, whose levels are $\equiv0\pmod k$, so
$X=\{x:b_x\in D\}$ is $c_k$-homogeneous of color $0$ and
$X\in\GG_k$. Let $S=T_X$, a pruned tree with
$[S]=\overline{b[X]}$ Polish and perfect (every committed branch is
split cofinally), in which every active node eventually has
infinitely many successors, and every nonempty relative open subset
of $[S]$ contains an active node (task (i) recurs at fresh levels).

Suppose $X\subseteq S_1\cup\dots\cup S_a\cup Y_1\cup\dots\cup
Y_b\cup F$ with $S_i$ selectors, $Y_j$ generators of bounded
branching ($|\suc_{T_{Y_j}}(s)|\leq|s|+1$) and $F$ finite. Then
$[S]$ is covered by the closures of the traces of the pieces, and it
suffices to show each closure is nowhere dense.

For a selector $S_i\subseteq A_{s_i}$: all pairs of $S_i$ separate
exactly at level $n_i=|s_i|$, so (as in Theorem~\ref{thm:H1PC}) any
two distinct points of $\overline{b[S_i\cap X]}$ have
$\Delta=n_i$. Every nonempty relative open subset of $[S]$ contains a
smaller basic open $[t']\cap[S]$ with $|t'|>n_i$ and at least two
points ($[S]$ is perfect); two such points have $\Delta\geq|t'|>n_i$,
so $[t']\cap[S]\not\subseteq\overline{b[S_i\cap X]}$: nowhere dense.

For a bounded-branching generator $Y_j$: the closure of
$b[Y_j\cap X]$ is contained in $[T_{Y_j}]$. If some nonempty
$[t]\cap[S]$ were contained in $[T_{Y_j}]$, take an active node
$u\supseteq t$ with infinitely many successors in $S$; each successor
extends to a branch of $[t]\cap[S]\subseteq[T_{Y_j}]$, so $u$ would
have infinitely many successors in $T_{Y_j}$, contradicting
$|\suc_{T_{Y_j}}(u)|\leq|u|+1$. Nowhere dense again.

Finally $b[F]$ is finite, hence nowhere dense in the perfect $[S]$.
The Baire category theorem for the Polish space $[S]$ gives a
contradiction, so $X\notin\PC$.
\end{proof}

\section{The Kat\v etov order of the chromatic ideals}\label{sec:katGk}

\begin{lemma}[level-stretching]\label{lem:stretching}
Let $h:\omega\to\omega$ be strictly increasing. There is an injection
$F:B\to B$ with $\Delta(F(p),F(q))=h(\Delta(p,q))$ for $p\neq q$.
\end{lemma}

\begin{proof}
Enumerate $B=\{p_0,p_1,\dots\}$ and construct the images
$q_i=F(p_i)$ by recursion, starting with an arbitrary $q_0\in B$.
Given $q_0,\dots,q_{N-1}$, put $\delta_i=\Delta(p_N,p_i)$, let
$m=\max_i\delta_i$, realized at $i_*$, and $d=h(m)$. If
$\delta_i=m$, then $\Delta(p_i,p_{i_*})\geq m$ by the ultrametric
inequality, so the inductive hypothesis gives
$\Delta(q_i,q_{i_*})=h(\Delta(p_i,p_{i_*}))\geq h(m)=d$: all such
$q_i$ share the prefix $q_{i_*}\restr d$ and forbid only finitely
many values at coordinate $d$. Choose a new value $a$ and, by
density, a new point $q_N\in B\cap[(q_{i_*}\restr d)\concat a]$.
Then $\Delta(q_N,q_i)=d=h(\delta_i)$ whenever $\delta_i=m$; and if
$\delta_i<m$, the ultrametric gives
$\Delta(p_i,p_{i_*})=\delta_i$, hence
$\Delta(q_i,q_{i_*})=h(\delta_i)<d$, and $q_N$, which follows
$q_{i_*}$ up to $d$, separates from $q_i$ exactly at
$h(\delta_i)$.
\end{proof}

\begin{theorem}\label{thm:kgeql}
If $k\geq\ell\geq2$ then $\GG_k\leq_{KB}\GG_\ell$, with an injective
witness.
\end{theorem}

\begin{proof}
Fix an injection $a:\{0,\dots,\ell-1\}\to\{0,\dots,k-1\}$ and build a
strictly increasing $h$ with $h(n)\equiv a(n\bmod\ell)\pmod k$
(every residue class modulo $k$ contains arbitrarily large numbers).
Let $F$ be the injection of Lemma~\ref{lem:stretching} and let
$f:\omega\to\omega$ be the induced injective point map
($b_{f(x)}=F(b_x)$). If $A$ is $c_k$-homogeneous of color $r$ and
$x,y\in f^{-1}[A]$ are distinct, then
$a(\Delta(x,y)\bmod\ell)=h(\Delta(x,y))\bmod k=r$; since $a$ is
injective, $f^{-1}[A]$ is $c_\ell$-homogeneous (of color
$a^{-1}(r)$) if $r$ is in the range of $a$, and has at most one point
otherwise. Preimages of finite unions of homogeneous sets and finite
sets stay in $\GG_\ell$, so $f$ witnesses
$\GG_k\leq_{KB}\GG_\ell$.
\end{proof}

\begin{corollary}\label{cor:3vs2}
$\GG_3\leq_{KB}\GG_2$ although $2\nmid3$: with $h(2m)=3m$,
$h(2m+1)=3m+1$ one has $h(n)\bmod3=n\bmod2$, so preimages of
homogeneous sets of colors $0,1$ are homogeneous of those colors,
while color $2$ has preimages of size at most one. In particular the
divisibility classification is \emph{false} for the Kat\v etov order:
inclusion preserves modular arithmetic
(Theorem~\ref{thm:divisibility}), Kat\v etov recodes it.
\end{corollary}

\begin{definition}\label{def:radial}
A \emph{radial reduction} of $\GG_k$ to $\GG_\ell$ is an injection
$F:B\to B$ for which there is a strictly increasing $h$ with
$\Delta(F(p),F(q))=h(\Delta(p,q))$, witnessing
$\GG_k\leq_K\GG_\ell$.
\end{definition}

\begin{lemma}[two-level lattice]\label{lem:lattice}
Let $n<m$ have distinct residues modulo $\ell$. There is
$X=\{x_{i,j}:i,j\in\omega\}$ with
$\Delta(x_{i,j},x_{i',j'})=n$ for $i\neq i'$,
$\Delta(x_{i,j},x_{i,j'})=m$ for $j\neq j'$, and
$X\notin\GG_\ell$.
\end{lemma}

\begin{proof}
Fix a node $u$ of length $n$. For each $i$ pick a distinct value
$v_i$ and a node $w_i\supseteq u\concat v_i$ of length $m$; then, by
density of $B$, pick branches
$b_{x_{i,j}}\in B\cap[w_i\concat z_{i,j}]$ with the $z_{i,j}$
distinct within each row, all points distinct. Rows separate at level
$n$, columns within a row at level $m$.

Let $H$ be $c_\ell$-homogeneous of color $r$. Pairs of $X$ separate
only at $n$ or $m$, and $n\not\equiv m\pmod\ell$. If
$r\equiv n$, then $H\cap X$ has at most one point per row; if
$r\equiv m$, then $H\cap X$ lies in a single row; for any other
color, $|H\cap X|\leq1$. Hence a finite union of homogeneous sets
covers finitely many full rows and finitely many extra points per
remaining row; adding a finite set still leaves infinitely many
points of some row uncovered. So $X\notin\GG_\ell$.
\end{proof}

\begin{proposition}[exact radial classification]\label{prop:radial}
There is a radial reduction of $\GG_k$ to $\GG_\ell$ if and only if
$k\geq\ell$.
\end{proposition}

\begin{proof}
Sufficiency is Theorem~\ref{thm:kgeql}, whose witness is radial. For
necessity, let $F$ be radial with stretching function $h$ and put
$S_r=\{n:h(n)\equiv r\pmod k\}$ for $r<k$. If some $S_r$ contained
$n<m$ with distinct residues modulo $\ell$, take the set $X$ of
Lemma~\ref{lem:lattice} for this pair: all separations within
$F[X]$ are $h(n)$ or $h(m)$, both $\equiv r\pmod k$, so $F[X]$ is
$c_k$-homogeneous of color $r$, hence in $\GG_k$; but
$F^{-1}[F[X]]=X\notin\GG_\ell$ ($F$ is injective), so $F$ would not
be a reduction. Hence each nonempty $S_r$ lies in a single residue
class modulo $\ell$. The sets $S_r$ partition $\omega$, so every one
of the $\ell$ residue classes modulo $\ell$ meets, and then equals
the class of, some $S_r$: the assignment $r\mapsto$ (class of $S_r$)
maps the at most $k$ nonempty indices \emph{onto} all $\ell$
classes, and therefore $k\geq\ell$.
\end{proof}

\begin{lemma}[necessary conditions]\label{lem:necessary}
If $f$ witnesses $\GG_k\leq_K\GG_\ell$, then every fiber of $f$ is in
$\GG_\ell$ and the image $f[\omega]$ is $\GG_k$-positive.
\end{lemma}

\begin{proof}
Each fiber $f^{-1}\{p\}$ is the preimage of the singleton $\{p\}$,
which lies in $\GG_k$; hence it lies in $\GG_\ell$. If
$f[\omega]\in\GG_k$, a finite cover of it by homogeneous sets and a
finite set would pull back to a cover of $\omega$ by members of
$\GG_\ell$, contradicting properness (Theorem~\ref{thm:H1PC}).
\end{proof}

\begin{conjecture}\label{conj:Gk}
$\GG_k\leq_K\GG_\ell\iff k\geq\ell$.
\end{conjecture}

\section{Branch accumulation: two orthogonal ranks}\label{sec:CB}

\begin{lemma}\label{lem:dictionary}
For every $X\subseteq\omega$, $[T^\infty_X]=(b[X])'$, the set of
accumulation points of $b[X]$ in $\omega^\omega$.
\end{lemma}

\begin{proof}
$z\in[T^\infty_X]$ iff every basic neighbourhood $[z\restr n]$
contains infinitely many points of $b[X]$, which in the metric space
$\omega^\omega$ is the definition of accumulation point.
\end{proof}

\begin{proposition}\label{prop:C}
The ideal $\Sc=\{X:[T^\infty_X]\text{ is countable}\}$ is proper,
tall, not a P-ideal, and $\mathbf\Pi^1_1$-complete.
\end{proposition}

\begin{proof}
$\Sc$ is an ideal: $(b[X\cup Y])'=(b[X])'\cup(b[Y])'$, and subsets
only shrink the accumulation set. It is proper:
$[T^\infty_\omega]=\omega^\omega$ is uncountable.

\emph{Tall.} Let $X$ be infinite. If $b[X]$ has an accumulation point
$z$, choose distinct $x_n\in X$ with $b_{x_n}\to z$; then
$Y=\{x_n:n\in\omega\}$ has $(b[Y])'=\{z\}$, so $Y\in\Sc$. Otherwise
$(b[X])'=\varnothing$ and $X\in\Sc$ itself.

\emph{Not a P-ideal.} $B$ is dense in $\omega^\omega$, so for every
$z\in\omega^\omega$ we may choose distinct points whose branches
converge to $z$. Fix a countable dense set
$\{z_j:j\in\omega\}\subseteq\omega^\omega$ and sets
$X_j=\{x_{j,i}:i\in\omega\}$ with $b_{x_{j,i}}\to z_j$; each
$X_j\in\Sc$. If $Y$ satisfied $X_j\subseteq^*Y$ for all $j$, then
$z_j\in(b[Y])'$ for all $j$, so $(b[Y])'$ is a closed set containing
a dense set: $(b[Y])'=\omega^\omega$ and $Y\notin\Sc$. Hence the
$X_j$ have no pseudounion in $\Sc$.

\emph{$\mathbf\Pi^1_1$-complete.} Upper bound: $X\notin\Sc$ iff
$[T^\infty_X]$ is uncountable iff $T^\infty_X$ contains a perfect
subtree $P$ (the tree of the perfect kernel of the closed set
$[T^\infty_X]$ works, and conversely $[P]\subseteq[T^\infty_X]$ is
uncountable); the condition ``$P$ is a perfect tree and every node of
$P$ has infinite trace in $X$'' is Borel in $(X,P)$, so the
complement of $\Sc$ is analytic. Hardness: fix an injection
$\langle\cdot,\cdot\rangle:\omega\times2\to\omega$ and, for a tree
$R$ on $\omega$, let $S(R)$ be the downward closure of
\[
 \bigl\{\bigl(\langle r(0),\varepsilon_0\rangle,\dots,
 \langle r(j-1),\varepsilon_{j-1}\rangle\bigr):
 r\in R\cap\omega^j,\ \varepsilon\in2^j\bigr\}.
\]
Then $[S(R)]$ is homeomorphic to $[R]\times2^\omega$: it is empty if
$R$ is well-founded, and contains a copy of $2^\omega$ (hence is
uncountable) if $R$ is ill-founded. The map is continuous: whether a
node lies in $S(R)$ reads one bit of $R$. Composing with the
Realization Lemma (Lemma~\ref{lem:realization}),
\[
 R\ \text{well-founded}\iff X_{S(R)}\in\Sc ,
\]
which reduces well-founded trees continuously to $\Sc$.
\end{proof}

\begin{theorem}[orthogonality]\label{thm:orthogonal}
$\Sc\subsetneq\HH$, and for every $\alpha<\omega_1$ there is
$X\in\Sc\setminus\HH_\alpha$, even with $[T^\infty_X]=\varnothing$.
\end{theorem}

\begin{proof}
$\Sc\subseteq\HH$: if $X\notin\HH$, then $T^\infty_X$ contains a
rooted Laver tree $L$ (Theorem~\ref{thm:nucleus} with
$\mathcal J=\Fin$), and $L$ contains a perfect binary subtree (every
node of $L$ has infinitely many children; keep two at each step), so
$[T^\infty_X]$ is uncountable and $X\notin\Sc$. The inclusion is
strict: realize a copy of the full binary tree hanging below a single
successor of the root, i.e.\ $X_S$ for
$S=\{\varnothing\}\cup\langle0\rangle\concat S_2$ with $S_2$ a copy
of $2^{<\omega}$ inside $\omega^{<\omega}$; then
$[S]$ is uncountable, so $X_S\notin\Sc$, while the root of $S$ has
one child, so $X_S\in\HH_1\subseteq\HH$.

Cofinal escape: the trees $S_\alpha$ of Theorem~\ref{thm:hierarchy}
(case $\mathcal J=\Fin$) are well-founded (by induction on the
construction: an infinite branch would enter a single hanging copy of
a previous tree), so $[S_\alpha]=\varnothing$ and
$X_{S_\alpha}\in\Sc$; but the root of $S_\alpha$ survives $\alpha$
derivatives, so $X_{S_\alpha}\notin\HH_\alpha$.
\end{proof}

\begin{remark}\label{rem:grid}
For $\alpha,\gamma<\omega_1$ let $\Sc_\alpha$ be the stratum of
Cantor--Bendixson rank $\leq\alpha$ of $[T^\infty_X]$ and
$\mathcal M_{\alpha,\gamma}=\Sc_\alpha\cap\HH_\gamma$. Contrary to
the compact-tree intuition, the fixed Cantor--Bendixson condition
need not be Borel. In fact, for every $\alpha<\omega_1$ and every
$1\leq\gamma<\omega_1$, the cell
$\mathcal M_{\alpha,\gamma}$ is $\mathbf\Pi^1_1$-hard.

To see this, let $S(R)$ be the tree used in the hardness proof of
Proposition~\ref{prop:C}, put
\[
 \widehat S(R)=\{\varnothing\}\cup
 \{\langle0\rangle\concat s:s\in S(R)\},
 \qquad X_R=X_{\widehat S(R)},
\]
and use the Realization Lemma. The map $R\mapsto X_R$ is continuous.
If $R$ is well-founded, then $[\widehat S(R)]=\varnothing$, so
$X_R\in\Sc_\alpha$, while the root of $\widehat S(R)$ has only one
successor, so $X_R\in\HH_1\subseteq\HH_\gamma$. If $R$ is
ill-founded, then $[\widehat S(R)]$ contains a copy of $2^\omega$,
so $X_R\notin\Sc_\alpha$. Thus well-founded trees reduce
continuously to $\mathcal M_{\alpha,\gamma}$. For $\gamma=0$ the
cell is simply $\Fin$, since $\HH_0=\Fin$.

Theorem~\ref{thm:orthogonal} still shows that the two coordinates are
independent in the strong sense that $\Sc$ escapes every
$\HH_\gamma$ already at Cantor--Bendixson rank $0$, but the
descriptive hardness is already present in each fixed cell with
$\gamma\geq1$. The fine structure of the grid --- strictness in each
coordinate, and whether it embeds pairs of ordinals into the
Kat\v etov order --- is posed as Problem~\ref{prob:grid}.
\end{remark}

\section{The quotient \texorpdfstring{$\Pow(\omega)/\HH$}{P(omega)/H}}\label{sec:quotient}

We write $\mathbb B=\Pow(\omega)/\HH$,
$\mathbb B_s=\Pow(A_s)/\HH(s)$, and
$\mathbb Q_{\HH}=(\HH^+,\leq_{\HH})$ with
$X\leq_{\HH}Y\iff X\setminus Y\in\HH$. A \emph{skeleton} over a rooted
Laver tree $L$ is a set $M=\{x_{s,n}:s,s\concat n\in L\}$ with
$x_{s,n}\in A_{s\concat n}$ distinct.

\begin{lemma}\label{lem:skeleton}
$X\notin\HH$ if and only if $X$ contains a skeleton over some rooted
Laver tree; every skeleton $M$ over $L$ satisfies $T^\infty_M=L$. In
particular skeletons are dense in $\mathbb Q_{\HH}$, literally (not
only modulo $\HH$).
\end{lemma}

\begin{proof}
If $M$ is a skeleton over $L$: for $u\in L$,
$M\cap A_u\supseteq\{x_{u,n}:u\concat n\in L\}$ is infinite ($L$ is
Laver); for $u\notin L$, the count of Lemma~\ref{lem:realization}
shows $M\cap A_u$ is finite. Hence $T^\infty_M=L$, and $M\notin\HH$
by Theorem~\ref{thm:nucleus}. Conversely, if $X\notin\HH$, take a
rooted Laver $L\subseteq T^\infty_X$
(Theorem~\ref{thm:nucleus}) and choose, by recursion along an
enumeration of the edges of $L$, a new point
$x_{s,n}\in X\cap A_{s\concat n}$ for each edge (possible, as each
$s\concat n\in T^\infty_X$): a skeleton inside $X$.
\end{proof}

\begin{theorem}[reduced-product decomposition]\label{thm:reduced}
For every node $s$: $X\in\HH(s)$ if and only if
$\{n:X\cap A_{s\concat n}\notin\HH(s\concat n)\}$ is finite, and
$[X]\mapsto\langle[X\cap A_{s\concat n}]\rangle_n$ is an isomorphism
\[
 \mathbb B_s\cong\prod_n\mathbb B_{s\concat n}\big/\Fin .
\]
\end{theorem}

\begin{proof}
\emph{Local equation.} If infinitely many sections
$X\cap A_{s\concat n}$ are $\HH(s\concat n)$-positive, each contains
(the trace of) a Laver tree rooted at $s\concat n$
(Theorem~\ref{thm:nucleus} localized); attaching them below $s$
yields a Laver tree rooted at $s$ inside $T^\infty_X$, so
$X\notin\HH(s)$. Conversely, a Laver tree rooted at $s$ has
infinitely many children of $s$, each rooting a Laver tree, so
$X\notin\HH(s)$ forces infinitely many positive sections.

\emph{Isomorphism.} The map is a Boolean homomorphism: operations are
computed cell by cell, and relative complements restrict correctly
($(A_s\setminus X)\cap A_{s\concat n}=A_{s\concat n}\setminus X$).
Well-definedness and injectivity are the local equation applied to
symmetric differences: $X\triangle Y\in\HH(s)$ iff all but finitely
many $n$ satisfy
$(X\cap A_{s\concat n})\triangle(Y\cap A_{s\concat n})
\in\HH(s\concat n)$, i.e.\ iff the images agree modulo $\Fin$.
Surjectivity: a sequence $\langle[Z_n]\rangle_n$ with
$Z_n\subseteq A_{s\concat n}$ is hit by $X=\bigcup_nZ_n$.
\end{proof}

\begin{corollary}\label{cor:selfref}
$\mathbb B\cong\mathbb B^\omega/\Fin$.
\end{corollary}

\begin{proof}
Theorem~\ref{thm:reduced} at $s=\varnothing$, plus local homogeneity
(Corollary~\ref{cor:homogeneity}):
$\mathbb B_{\langle n\rangle}\cong\mathbb B$ for every $n$.
\end{proof}

\begin{theorem}\label{thm:regular}
The map $\Phi([B]_{\Fin})=
[\bigcup_{n\in B}A_{\langle n\rangle}]_{\HH}$ is a unital embedding of
$\Pow(\omega)/\Fin$ into $\mathbb B$, and
$\pi([X])=
[\{n:X\cap A_{\langle n\rangle}\notin\HH(\langle n\rangle)\}]_{\Fin}$
is a forcing projection with $\pi\circ\Phi=\mathrm{id}$. The copy is
regular and $\mathbb Q_{\HH}$ factors over
$(\Pow(\omega)/\Fin)^+$.
\end{theorem}

\begin{proof}
$\Phi$ preserves the Boolean operations because the first-level cones
partition $\omega$. Its kernel is $\Fin$: a finite union of cones is
in $\HH$ (each cone $A_{\langle n\rangle}$ has trace with a single
child of the root), while an infinite union contains a rooted Laver
tree (full cones below infinitely many children). Hence $\Phi$ is an
embedding, clearly unital.

$\pi$ is well defined on classes and order-preserving: if
$X\setminus Y\in\HH$, then, by the local equation
(Theorem~\ref{thm:reduced}), positive sections of $X$ are positive
for $Y$ with finitely many exceptions. And
$\pi(\Phi([B]))=[\{n\in B:A_{\langle n\rangle}\notin
\HH(\langle n\rangle)\}]=[B]$, since every cell is positive for its
own localized ideal.

\emph{Projection property.} Let $X\in\HH^+$ and
$[B]\leq\pi([X])$ with $B$ infinite. Put
$Y=X\cap\bigcup_{n\in B}A_{\langle n\rangle}$. Infinitely many
$n\in B$ have $X\cap A_{\langle n\rangle}$ positive (as
$B\subseteq^*\pi$-set), and $Y$ has the same sections there, so $Y$
is positive by the local equation; moreover $Y\subseteq X$ and
$\pi([Y])\leq[B]$. This is the definition of a forcing projection,
and it also shows regularity of the copy: given $[B]\leq\pi([X])$,
the condition $Y$ is below both $X$ and
$\Phi([B])$ (literally $Y\subseteq\bigcup_{n\in B}
A_{\langle n\rangle}$), so $\pi([X])$ is a reduction of $[X]$ to the
copy, and every maximal antichain of $(\Pow(\omega)/\Fin)^+$ remains
maximal below the copy. Hence $\mathbb Q_{\HH}$ factors over
$(\Pow(\omega)/\Fin)^+$.
\end{proof}

\begin{theorem}\label{thm:sigmaclosed}
$\mathbb Q_{\HH}$ is $\sigma$-closed; hence it is proper, its
completion is $\omega$-distributive, and it adds no reals.
\end{theorem}

\begin{proof}
Let $[X_0]\geq[X_1]\geq\cdots$ be a descending sequence in
$\mathbb Q_{\HH}$ and put $Y_k=\bigcap_{j\leq k}X_j$. For every
$j<k$ we have $[X_k]\leq[X_j]$, hence $X_k\setminus X_j\in\HH$, and
therefore
\[
 X_k\setminus Y_k=\bigcup_{j<k}(X_k\setminus X_j)\in\HH .
\]
Thus $[Y_k]=[X_k]$, and after replacing $X_k$ by $Y_k$ we may assume
$X_{k+1}\subseteq X_k$ literally. By the local equation
(Theorem~\ref{thm:reduced}), each $X_k$ has infinitely many positive
first-level sections; recursively choose \emph{distinct} indices
$n_k$ with $X_k\cap A_{\langle n_k\rangle}$ positive, and skeletons
$M_k\subseteq X_k\cap A_{\langle n_k\rangle}$
(Lemma~\ref{lem:skeleton}, localized). The diagonal
$M=\bigcup_kM_k$ is positive: its trace contains a rooted Laver tree
(infinitely many children of the root, each carrying a localized
Laver tree). For every $m$,
$M\setminus X_m\subseteq\bigcup_{k<m}M_k$, a finite union of subsets
of single cones, which lies in $\HH$; so $[M]\leq[X_m]$ for all $m$.
\end{proof}

\begin{theorem}[constant-trace antichain]\label{thm:antichain}
Let $X\in\HH^+$ and let $L\subseteq T^\infty_X$ be rooted Laver. There
is $\{M_r:r\in2^\omega\}\subseteq\Pow(X)$ with $T^\infty_{M_r}=L$ for
all $r$ and $M_r\cap M_q$ finite for $r\neq q$. Consequently below
every positive condition there is an antichain of size $\cfrak$, the
cellularity and density of $\mathbb B$ are $\cfrak$, and
$X\mapsto T^\infty_X$ recognizes neither the order nor the
compatibility of the quotient.
\end{theorem}

\begin{proof}
Enumerate the edges of $L$ as
$\langle(s_i,s_i\concat m_i):i\in\omega\rangle$. Since the edge cells
may be nested, choose the auxiliary sets simultaneously: enumerate
$\omega\times\omega$ and, at each step $(i,j)$, pick a new point
$p_i(j)\in X\cap A_{s_i\concat m_i}$ different from all points chosen
so far (possible: each edge cell meets $X$ in an infinite set, and
only finitely many points are in use). This yields pairwise disjoint
infinite sets $P_i=\{p_i(j):j\in\omega\}\subseteq X$, with all points
globally distinct. Take an eventually different family
$\{f_r:r\in2^\omega\}$, e.g.\ $f_r(i)=$ (a code of) $r\restr i$, so
that $r\neq q$ implies $f_r(i)\neq f_q(i)$ for all large $i$. Put
\[
 M_r=\{p_i(f_r(i)):i\in\omega\} .
\]
Each $M_r$ picks exactly one point in each edge cell of $L$, so it is
a skeleton over $L$ and $T^\infty_{M_r}=L$
(Lemma~\ref{lem:skeleton}). For $r\neq q$, the intersection
$M_r\cap M_q\subseteq\{p_i(f_r(i)):f_r(i)=f_q(i)\}$ is finite. Two
distinct $M_r,M_q$ are incompatible in $\mathbb Q_{\HH}$: a common
lower bound would be almost contained in the finite set
$M_r\cap M_q$, hence in $\HH$. This gives an antichain of size
$\cfrak$ below $[X]$; cellularity and density $\cfrak$ follow (the
upper bound is the cardinality of $\Pow(\omega)$). Since all the
$M_r$ share the same trace tree, the map $X\mapsto T^\infty_X$
recognizes neither order nor compatibility.
\end{proof}

\begin{remark}[it is not Laver forcing]\label{rem:notlaver}
Three obstructions. (i) If $s\neq\varnothing$ then $A_s\in\HH$, so no
condition can fix a stem. (ii) By Theorem~\ref{thm:antichain}, fixing
the whole trace tree leaves continuum many pairwise incompatible
conditions: the decoration of points along the edges is an essential
part of a condition. (iii) Laver forcing adds a dominating real
\cite{Laver,GoldsternEtAl}, while $\mathbb Q_{\HH}$ adds no reals
(Theorem~\ref{thm:sigmaclosed}); accordingly, $\HH$ is not the
Hru\v s\'ak--Zapletal trace ideal of the Laver $\sigma$-ideal. The
quotient by that trace ideal factors with Laver forcing and therefore
adds a dominating real
\cite[Example~3.7 and Proposition~3.8]{HrusakZapletal}. For other
ideals directly associated with Laver and Miller trees, see
\cite{CieslakMartinez}.
\end{remark}

\begin{corollary}[CH]\label{cor:CH}
Under CH, $\mathbb Q_{\HH}$ is forcing-equivalent to
$(\Pow(\omega)/\Fin)^+$.
\end{corollary}

\begin{proof}
The separative quotient of $\mathbb Q_{\HH}$ is
$\mathbb B\setminus\{0\}$, separative as the nonzero part of a
Boolean algebra; atomless by Theorem~\ref{thm:antichain};
$\omega_1$-closed by Theorem~\ref{thm:sigmaclosed} (descending chains
of countable limit length reduce to cofinal $\omega$-subsequences);
and of size $\cfrak=\omega_1$ under CH (at most $\cfrak$ by
cardinality, at least $\cfrak$ by the antichain). Fact 2.7(b) of
\cite{Kurilic} states that, under $\cfrak=\omega_1$, every atomless
separative $\omega_1$-closed pre-order of size $\omega_1$ is
forcing-equivalent to $(\Pow(\omega)/\Fin)^+$.
\end{proof}

\begin{remark}
By Corollary~\ref{cor:CH} no ZFC invariant can separate
$\mathbb Q_{\HH}$ from $(\Pow(\omega)/\Fin)^+$; the right question is
whether ZFC \emph{proves} the equivalence, and the pertinent invariant
is the distributivity number. See Problem~\ref{prob:hQ}.
\end{remark}

\section{Open problems}\label{sec:problems}

\begin{problem}[central]\label{prob:central}
Decide whether $\GG_2\leq_K\GG_3$. By Proposition~\ref{prop:radial}
there is no radial witness. A negative answer gives
Conjecture~\ref{conj:Gk} in its first case and a strict
$(\omega,\geq)$-chain in the Kat\v etov order; a positive answer
collapses all Kat\v etov types of the family. The missing piece for
the negative is a canonization lemma; the strictness technique of
\cite{FKK} is the natural candidate.
\end{problem}

\begin{problem}\label{prob:limits}
The remaining sides of the triangle of
Remark~\ref{rem:threelimits}: does $\Hw\leq_K\Fwp$ hold? Does
$\Hw\leq_K\Fw$? By Theorem~\ref{thm:katH} both questions are
equivalent to the existence of isomorphic copies. A positive answer
to the second would yield $\Hw<_K\Fw$; together with a negative answer
to the first it would yield the strict chain
$\Fwp<_K\Hw<_K\Fw$ of amalgamations. This is motivated by, but does
not itself settle a new case of, the transfinite conjecture of
\cite{Kwela} (Conjecture 5.3); see Remark~\ref{rem:kwelascope}.
\end{problem}

\begin{problem}\label{prob:cf}
For arbitrary $f:\omega\to r$ and $c_f(\{x,y\})=f(\Delta(x,y))$,
classify the ideals generated by the $c_f$-homogeneous sets, by
inclusion and by Kat\v etov, in terms of the (recurrent) fibers of
$f$. Theorems \ref{thm:divisibility} and \ref{thm:kgeql} are the first
two cases.
\end{problem}

\begin{problem}\label{prob:random}
Locate the random ideal $\mathcal R$ of \cite{HrusakEtAl} with respect
to the family $\GG_k$. Since $\GG_2$ is the Kat\v etov maximum of the
family (Theorem~\ref{thm:kgeql}), one direction reduces to comparing
with $\GG_2$.
\end{problem}

\begin{problem}\label{prob:hQ}
Is $h(\mathbb Q_{\HH})=h$? Does ZFC prove that $\mathbb Q_{\HH}$ is
forcing-equivalent to $(\Pow(\omega)/\Fin)^+$? If ZFC proved
$h(\mathbb Q_{\HH})=\omega_1$, then MA$+\neg$CH would separate the two
forcings and, with Corollary~\ref{cor:CH}, the equivalence would be
independent.
\end{problem}

\begin{problem}\label{prob:residual}
Compute the quotient forcing of $\mathbb Q_{\HH}$ over
$(\Pow(\omega)/\Fin)^+$ induced by the projection of
Theorem~\ref{thm:regular}. The two-step analysis of \cite{Kurilic} is
the template.
\end{problem}

\begin{problem}\label{prob:DJ}
For Borel $\mathcal J$: is the hierarchy
$\langle\HH^{\mathcal J}_\alpha\rangle$ strict in the Kat\v etov
order, not only under inclusion? Can the separation rank of
$\HH^{\mathcal J}_\alpha$ be computed from those of $\mathcal J$ and
$\alpha$? Is there reflection:
$\HH^{\mathcal J}\leq_K\HH^{\mathcal K}$ implies
$\mathcal J\leq_K\mathcal K$ under natural hypotheses?
\end{problem}

\begin{problem}\label{prob:grid}
Is the grid $\mathcal M_{\alpha,\gamma}=\Sc_\alpha\cap\HH_\gamma$ of
Remark~\ref{rem:grid} strict in each coordinate? Does it realize an
embedding of pairs of ordinals into the Kat\v etov order (not only
under inclusion)?
\end{problem}

\medskip
\noindent\textbf{Statement on the use of AI-assisted tools.}
During the preparation of this manuscript, the author used Claude
(Anthropic) and Codex (OpenAI) to assist in writing computer code for
finite examples and in translating, revising, reorganizing, and drafting
parts of the text. All AI-assisted output was critically reviewed,
edited, and independently verified by the author. The mathematical ideas
and original results are the author's own, and the author assumes full
responsibility for the final version, including its accuracy, originality,
and integrity.

\end{document}